\documentclass[11pt,leqno]{article}
\usepackage[margin=1in]{geometry} 
\usepackage{amssymb,amsfonts,amsmath,bbm,mathrsfs,stmaryrd,mathtools}
\usepackage{xcolor}
\usepackage{url}

\usepackage{extarrows}

\usepackage[shortlabels]{enumitem}
\usepackage{tensor}

\usepackage{xr}
\usepackage[T1]{fontenc}

\usepackage[colorlinks,
linkcolor=black!75!red,
citecolor=blue,
pdftitle={},
pdfproducer={pdfLaTeX},
pdfpagemode=None,
bookmarksopen=true,
bookmarksnumbered=true,
backref=page]{hyperref}

\usepackage{tikz}
\usetikzlibrary{arrows,calc,decorations.pathreplacing,decorations.markings,intersections,shapes.geometric,through,fit,shapes.symbols,positioning,decorations.pathmorphing}

\usepackage{braket}

\usepackage[amsmath,thmmarks,hyperref]{ntheorem}
\usepackage{cleveref}

\creflabelformat{enumi}{#2#1#3}

\crefname{section}{Section}{Sections}
\crefformat{section}{#2Section~#1#3} 
\Crefformat{section}{#2Section~#1#3} 

\crefname{subsection}{\S}{\S\S}
\AtBeginDocument{%
  \crefformat{subsection}{#2\S#1#3}%
  \Crefformat{subsection}{#2\S#1#3}%
}

\crefname{subsubsection}{\S}{\S\S}
\AtBeginDocument{%
  \crefformat{subsubsection}{#2\S#1#3}%
  \Crefformat{subsubsection}{#2\S#1#3}%
}

\theoremstyle{plain}

\newtheorem{lemma}{Lemma}[section]
\newtheorem{proposition}[lemma]{Proposition}
\newtheorem{corollary}[lemma]{Corollary}
\newtheorem{theorem}[lemma]{Theorem}

\theoremstyle{plain}
\theoremnumbering{Alph}
\newtheorem{theoremN}{Theorem}
\newtheorem{propositionN}[theoremN]{Proposition}

\theoremstyle{plain}
\theorembodyfont{\upshape}
\theoremsymbol{\ensuremath{\blacklozenge}}

\newtheorem{definition}[lemma]{Definition}
\newtheorem{example}[lemma]{Example}
\newtheorem{examples}[lemma]{Examples}
\newtheorem{remark}[lemma]{Remark}
\newtheorem{remarks}[lemma]{Remarks}

\newtheorem{notation}[lemma]{Notation}

\crefname{definition}{definition}{definitions}
\crefformat{definition}{#2definition~#1#3} 
\Crefformat{definition}{#2Definition~#1#3} 

\crefname{ex}{example}{examples}
\crefformat{example}{#2example~#1#3} 
\Crefformat{example}{#2Example~#1#3} 

\crefname{exs}{example}{examples}
\crefformat{examples}{#2example~#1#3} 
\Crefformat{examples}{#2Example~#1#3} 

\crefname{remark}{remark}{remarks}
\crefformat{remark}{#2remark~#1#3} 
\Crefformat{remark}{#2Remark~#1#3} 

\crefname{remarks}{remark}{remarks}
\crefformat{remarks}{#2remark~#1#3} 
\Crefformat{remarks}{#2Remark~#1#3} 

\crefname{convention}{convention}{conventions}
\crefformat{convention}{#2convention~#1#3} 
\Crefformat{convention}{#2Convention~#1#3} 

\crefname{notation}{notation}{notations}
\crefformat{notation}{#2notation~#1#3} 
\Crefformat{notation}{#2Notation~#1#3} 

\crefname{table}{table}{tables}
\crefformat{table}{#2table~#1#3} 
\Crefformat{table}{#2Table~#1#3}

\crefname{lemma}{lemma}{lemmas}
\crefformat{lemma}{#2lemma~#1#3} 
\Crefformat{lemma}{#2Lemma~#1#3} 

\crefname{proposition}{proposition}{propositions}
\crefformat{proposition}{#2proposition~#1#3} 
\Crefformat{proposition}{#2Proposition~#1#3} 

\crefname{propositionN}{proposition}{propositions}
\crefformat{propositionN}{#2proposition~#1#3} 
\Crefformat{propositionN}{#2Proposition~#1#3} 

\crefname{corollary}{corollary}{corollaries}
\crefformat{corollary}{#2corollary~#1#3} 
\Crefformat{corollary}{#2Corollary~#1#3} 

\crefname{corollaryN}{corollary}{corollaries}
\crefformat{corollaryN}{#2corollary~#1#3} 
\Crefformat{corollaryN}{#2Corollary~#1#3} 

\crefname{theorem}{theorem}{theorems}
\crefformat{theorem}{#2theorem~#1#3} 
\Crefformat{theorem}{#2Theorem~#1#3} 

\crefname{theoremN}{theorem}{theorems}
\crefformat{theoremN}{#2theorem~#1#3} 
\Crefformat{theoremN}{#2Theorem~#1#3} 

\crefname{enumi}{}{}
\crefformat{enumi}{#2#1#3}
\Crefformat{enumi}{#2#1#3}

\crefname{assumption}{assumption}{Assumptions}
\crefformat{assumption}{#2assumption~#1#3} 
\Crefformat{assumption}{#2Assumption~#1#3} 

\crefname{construction}{construction}{Constructions}
\crefformat{construction}{#2construction~#1#3} 
\Crefformat{construction}{#2Construction~#1#3} 

\crefname{question}{question}{Questions}
\crefformat{question}{#2question~#1#3} 
\Crefformat{question}{#2Question~#1#3} 

\crefname{equation}{}{}
\crefformat{equation}{(#2#1#3)} 
\Crefformat{equation}{(#2#1#3)}

\numberwithin{equation}{section}

\theoremstyle{nonumberplain}
\theoremsymbol{\ensuremath{\blacksquare}}

\newtheorem{proof}{Proof}

\newcommand\bF{{\mathbb F}}
\newcommand\bG{{\mathbb G}}
\newcommand\bH{{\mathbb H}}

\newcommand\bK{{\mathbb K}}
\newcommand\bL{{\mathbb L}}

\newcommand\bQ{{\mathbb Q}}
\newcommand\bR{{\mathbb R}}

\newcommand\bT{{\mathbb T}}

\newcommand\bZ{{\mathbb Z}}

\newcommand\cA{{\mathcal A}}

\newcommand\cE{{\mathcal E}}
\newcommand\cF{{\mathcal F}}
\newcommand\cG{{\mathcal G}}
\newcommand\cH{{\mathcal H}}

\newcommand\cM{{\mathcal M}}

\newcommand\cS{{\mathcal S}}

\newcommand\mc{\mathrm{c}}

\DeclareMathOperator{\id}{id}

\DeclareMathOperator{\im}{\mathrm{im}}

\DeclareMathOperator{\Supp}{\mathrm{Supp}}

\newcommand{\cat}[1]{\textsc{#1}}

\newcommand\spr[1]{\cite[\href{https://stacks.math.columbia.edu/tag/#1}{Tag {#1}}]{stacks-project}}
\newcommand{\qedhere}{\mbox{}\hfill\ensuremath{\blacksquare}}

\newcommand{\xrightarrowdbl}[2][]{%
  \xrightarrow[#1]{#2}\mathrel{\mkern-14mu}\rightarrow
}

\title{Proper actions and supported-section-valued cohomology}
\author{Alexandru Chirvasitu}

\begin{document}

\date{}

\newcommand{\Addresses}{{
  \bigskip
  \footnotesize

  \textsc{Department of Mathematics, University at Buffalo}
  \par\nopagebreak
  \textsc{Buffalo, NY 14260-2900, USA}  
  \par\nopagebreak
  \textit{E-mail address}: \texttt{achirvas@buffalo.edu}


}}

\maketitle

\begin{abstract}
  Consider a proper action of $\mathbb{Z}^d$ on a smooth (perhaps non-paracompact) manifold $M$. The $p^{th}$ cohomology $H^p(\mathbb{Z}^d,\ \Gamma_{\mathrm{c}}(\mathcal{F}))$ valued in the space of compactly-supported sections of a natural sheaf $\mathcal{F}$ on $M$ (such as those of smooth function germs, smooth $k$-form germs, etc.) vanishes for $p\ne d$ (the cohomological dimension of $\mathbb{Z}^d$) and, at $d$, equals the space of compactly-supported sections of the descent ($\mathbb{G}$-invariant push-forward) $\mathcal{F}/\mathbb{Z}^d$ to the orbifold quotient $M/\mathbb{Z}^d$. We prove this and analogous results on $\mathbb{Z}^d$ cohomology valued in $\Phi$-supported sections of an equivariant appropriately soft sheaf $\mathcal{F}$ in a broader context of $\mathbb{Z}^d$-actions proper with respect to a paracompactifying family of supports $\Phi$, in the sense that every member of $\Phi$ has a neighborhood small with respect to the action in Palais' sense.
\end{abstract}

\noindent {\em Key words:
  Lyndon-Hochschild-Serre spectral sequence;
  Tychonoff space;
  family of supports;
  group cohomology;
  paracompactifying;
  proper action;
  sheaf cohomology;
  soft sheaf

}

\vspace{.5cm}

\noindent{MSC 2020:
  20J06; 
  18G40; 
  55T10; 
  54D20; 
  54D15; 
  54B15; 
  55N30; 
  18F20 

}


\section*{Introduction}

One piece of motivation for the present paper lies in the types of smooth-dynamics considerations pertinent to \cite{2408.15053v3}. The focus there is on smooth flows $M\times \bR\xrightarrow{\alpha} M$ for compact smooth $M$, the central concern being the (non)invertibility of the integral operators
\begin{equation*}
  \frac 1s\int_0^s\alpha_t\ \mathrm{d}t
  ,\quad
  s\in \bR_{>0}
\end{equation*}
on $C^{\infty}(M)$ (symbolically conflating the diffeomorphisms $\alpha_t$ and the automorphisms they induce on the topological vector space $C^{\infty}(M)$). A brief and rough sketch of how \cite[Theorem 3.10(1)]{2408.15053v3} addresses the problem when the flow has at least one non-periodic \emph{locally closed} orbit $\alpha(p,\bR)\subset M$ (i.e. \cite[Problem 2.7.1]{eng_top_1989} one that is open in its closure) is as follows:
\begin{enumerate}[(a),wide]
\item\label{item:add} identify the flow on said orbit with the usual translation flow $\bR\times \bR\xrightarrow{+}\bR$ on the real line;

\item\label{item:cpct.supp} note that the assumed local closure implies that all smooth compactly-supported functions on $\bR$ restrict from smooth functions on $M$;

\item\label{item:cpct.supp.all} observe that for any fixed $s>0$ the smooth compactly-supported $f$ are precisely the functions of the form
  \begin{equation*}
    f=F(\bullet+s)-F
    ,\quad
    F\in C^{\infty}(\bR)
    ,\quad
    F\text{ $s$-periodic close to either }\pm\infty;
  \end{equation*}

\item\label{item:cpct.supp.some} whereas by contrast, for (Lebesgue-)almost all $s$ the functions
  \begin{equation*}
    f\in C_{\mc}(\bR)\cap\left(\im\beta_s\bigg|_{\bR}\right)    
  \end{equation*}
  are of the form $F(\bullet+s)-F$ for $F\in C_{\mc}^{\infty}(\bR)$ (compactly-supported, i.e. eventually \emph{vanishing} at $\pm\infty$). 
\end{enumerate}
In terms of the $\bZ$-action on $\bR$ given by translation by $s$, associating $F(\bullet+s)-F$ to $F\in C^{\infty}(\bR)$ is the \emph{coboundary} map \cite[Application 6.5.5]{weib_halg}
\begin{equation*}
  C^{\infty}(\bR)
  \cong
  \text{0-cochains }
  C^0\left(\bZ,\ C^{\infty}(\bR)\right)  
  \xrightarrow{\quad\quad}
  B^1\left(\bZ,\ C^{\infty}(\bR)\right)
  :=
  \text{1-coboundaries}
\end{equation*}
familiar from group cohomology. One way to cast the above-noted discrepancy between \Cref{item:cpct.supp.all} and \Cref{item:cpct.supp.some} is as the non-triviality of the cohomology group $H^1\left(\bZ,\ C_{\mc}^{\infty}(\bR)\right)$. Placing such cohomology-non-vanishing results in some surrounding context is what forms the object of the paper.

In parallel to the translation action of $\bZ$ on $\bR$, most actions discussed below will be 
\begin{itemize}[wide]
\item occasionally \emph{free} ($xg=x\Rightarrow g=1$);

\item and very often \emph{proper} in the sense of \cite[\S III.4.1, Definition 1]{bourb_top_en_1}: the continuous map
  \begin{equation*}
    X\times \bG
    \ni
    (x,g)
    \xmapsto{\quad\varphi\quad}
    (xg,x)
    \in
    X^2
  \end{equation*}
  is \emph{proper}, i.e. \cite[\S I.10.1, Definition 1]{bourb_top_en_1}
  \begin{equation*}
    \forall\text{ topological }Z
    \quad:\quad
    X\times\bG\times Z
    \xrightarrow[\quad\text{is closed}\quad]{\quad\varphi\times\id_Z\quad}
    X\times X\times Z.
  \end{equation*}
  Equivalently \cite[\S I.10.2, Theorem 1]{bourb_top_en_1}, a map is proper if it is closed and point preimages are compact (Hausdorff or not). 
\end{itemize}

The above-mentioned non-triviality $H^1\left(\bZ,\ C_{\mc}^{\infty}(\bR)\right)\ne 0$ is an instance of \Cref{th:free.prop.zd.cpct} below; we recall the germane sheaf-theoretic notions (\emph{softness} and \emph{support families} \cite[Definitions I.6.1 and II.9.1]{bred_shf_2e_1997}) in \Cref{se:cpct}.

\begin{theoremN}\label{thn:intro.cpct.supp.zd}
  Let $X\times \bZ^d\to X$ be a proper action on a locally compact Hausdorff space and $\cF$ an equivariant abelian sheaf on $X$, soft with respect to the support family generated by
  \begin{equation*}
    \mc\bZ^d
    :=
    \left\{K\bZ^d\ :\ K\in \mc:=\text{compact subsets of }X\right\}
  \end{equation*}
  we have
  \begin{equation*}
    H^p\left(\bZ^d,\ \Gamma_{\mc}(\cF)\right)
    \cong
    \begin{cases}
      \Gamma_{\mc}\left(X/\bZ^d,\ (\pi_*\cF)^{\bZ^d}\right)
      &\text{if }p=d\\
      0
      &\text{otherwise},
    \end{cases}
  \end{equation*}
  where $X\xrightarrowdbl{\pi}X/\bZ^d$ is the quotient map.  \qedhere
\end{theoremN}

The recovery of $d$, the \emph{cohomological dimension} \cite[\S 9.9, p.591, Definition]{rot} of the acting group $\bZ^d$ as the only cohomology non-vanishing degree is reminiscent of the familiar computation of compactly-supported cohomology for $\bR^d$ (e.g. \cite[Corollary 4.7.1]{bt_forms}). The cohomological non-triviality recorded by \Cref{thn:intro.cpct.supp.zd} contrasts with a number of familiar \emph{a}cyclicity phenomena:
\begin{itemize}[wide]
\item For a free action $X\times \bG\to X$ of a \emph{compact} group on a locally compact Hausdorff space the analogous higher cohomology groups
  \begin{equation*}
    H^p\left(B_{\bG},\ \cH^0_{\mc}\left(X,\pi^*\cG\right)\right)
    ,\quad
    p\ge 1
    ,\quad
    B_{\bG}:=\text{\emph{classifying space} \cite[Definition 7.2.7]{hjjm_bdle} of $\bG$}
  \end{equation*}
  valued in the \emph{Leray sheaf} \cite[Definition IV.4.1]{bred_shf_2e_1997} attached to $X\xrightarrowdbl{\pi}X/\bG$ vanish under suitable softness assumptions on the sheaf $\cG$ on $X/\bG$, as observed more expansively in \Cref{re:zd.vs.cpct.lie}.
  
\item On the other hand, \Cref{res:soft.not.desc}\Cref{item:res:soft.not.desc:bo} records how the proof of \cite[Lemma 2.4]{zbMATH00825443} can be adapted to prove that
  \begin{equation*}
    H^p(\bG,\ \Gamma(\cF))=0
    ,\quad
    \forall p\ge 1
  \end{equation*}
  whenever the discrete group $\bG$ acts properly on a metrizable, locally compact, \emph{$\sigma$-compact} space $X$ (the latter condition means \cite[Problem 17I]{wil_top} $X$ is a countable union of compact subspaces) and $\cF$ is an abelian $\bG$-equivariant sheaf of modules over real-valued continuous germs.
\end{itemize}

Following up on the latter point, \Cref{pr:bo.gen} generalizes \cite[Lemma 2.4]{zbMATH00825443}, dispensing among other things with the requirement that the sheaves in question be $\bR$-linear.

\begin{propositionN}\label{prn:bo.gen}
  Let $\Phi\subseteq 2^X$ be a paracompactifying family of supports on a Hausdorff space $X$, $\bG$-invariant under a Bourbaki-proper action $X\times \bG\to X$ by a discrete group. Assume that
  \begin{itemize}[wide]
  \item the action is \emph{$\Phi$-proper} in the sense of \Cref{def:transv.2.supp}\Cref{item:def:transv.2.supp:phi};

  \item the abelian $\bG$-equivariant sheaf $\cF$ on $X$ is soft with respect to the family $\Supp(\Phi\bG)$ of supports generated by
    \begin{equation*}
      \left\{F\bG\ :\ F\in\Phi\right\};
    \end{equation*}

  \item the invariant push-forward $(\pi_*\cF)^{\bG}$ along $X\xrightarrowdbl{\pi}X/\bG$ is soft with respect to the family of supports generated by (the closed sets: \Cref{le:pass2quot.prpr}) $\pi(F)$, $F\in \Phi$;

  \item and
    \begin{equation*}
      H^i(\bG_x,\ \cF_x)=0
      ,\quad
      \forall i>0
      ,\quad
      \forall x\in X
    \end{equation*}
    for the respective isotropy groups
    \begin{equation*}
      \bG_x:=\left\{g\in \bG\ :\ xg=x\right\}.
    \end{equation*}
  \end{itemize}
  We then have $H^i(\bG,\ \Gamma_{\Supp(\Phi\bG)}(\cF))=0$ for $i>0$.  \qedhere
\end{propositionN}

\Cref{thn:intro.cpct.supp.zd} is in turn a particular case of a more general result (\Cref{th:free.prop.zd.gen}) abstracting away from the family of compact subsets to arbitrary \emph{paracompactifying} \cite[Definition I.6.1]{bred_shf_2e_1997} families of supports. Building up the background for that discussion requires some auxiliary results on action properness, and in passing we revisit the distinction between what some of the dynamical-systems literature terms \emph{Bourbaki properness} (the notion recalled above \cite[\S III.4.1, Definition 1]{bourb_top_en_1}) and \emph{Palais properness} \cite[Definition 1.2.2]{palais_slice-nc}.

It is known that Palais properness implies the Bourbaki version but not conversely \cite[Example 2.11.13]{bs_stab_2002}. \Cref{pr:bbt.ext} builds on those examples:

\begin{propositionN}\label{prn:bourb.pal}
  Let $\bG$ be a locally compact Hausdorff group. If all Bourbaki-proper (free) $\bG$-actions on $T_{3\frac 12}$ spaces are Palais-proper then $\bG$ contains no copy of $\bZ$ as a closed subgroup.

  Equivalently, $\bG$ is a union of compact open subgroups.  \qedhere
\end{propositionN}

\subsection*{Acknowledgements}

I am grateful for insightful and instructive comments on various portions of this work, at several stages, from A. Badzioch, S. Hurder, M. Kapovich, J.M. Lee, K.H. Neeb, A. Schmeding and F. Wagemann.


\section{Cohomology valued in compactly-supported sections}\label{se:cpct}

We assume some of the standard background on \emph{sheaves} as covered, say, in \cite{ks_shv-mfld,bred_shf_2e_1997} and countless other sources. Some notation:
\begin{itemize}[wide]
\item $\cS h_X$, $\cA b_X$ and $\tensor*[_{\cA}]{\cM od}{}$ denote, respectively, the categories of plain sheaves, abelian sheaves and $\cA$-modules on $X$ for a sheaf of rings $\cA$.
  
\item More generally, if $\bG$ is a group acting on $X$, $\cS h_X^{\bG}$ ($\cA b_X^{\bG}$) is the category of (abelian) \emph{$\bG$-sheaves} \cite[\S 5.1]{zbMATH03192990} (or \emph{$\bG$-equivariant sheaves} \cite[\S 0.2]{bl_eq-shv}) on $X$. 

\item Similarly, if $\cA\in \cA b_X$ is a $\bG$-sheaf of rings then $\tensor*[_{\cA}]{\cM od}{^{\bG}}$ is the category of equivariant modules over it.
\end{itemize}

For equivariant $\cF\in \tensor*[_{\cA}]{\cM od}{^{\bG}}$ the group $\bG$ will operate on the section space $\Gamma(\cF)$. More generally, suppose $\Phi\subseteq 2^X$ is a \emph{family of supports} \cite[Definition I.6.1]{bred_shf_2e_1997} on $X$: a family of closed subsets, closed under finite unions and containing, along with a set, all of its closed subsets. If $\Phi$ is $\bG$-invariant in the guessable sense that
\begin{equation*}
  A\in \Phi
  \xRightarrow{\quad}
  Ag\in \Phi
  ,\quad g\in \bG
\end{equation*}
then $\bG$ will operate on the space $\Gamma_{\Phi}(\cF)$ of $\Phi$-supported sections. This applies in particular to
\begin{equation}\label{eq:cl.c}
  \Phi:=\mathrm{cl}:=\left\{\text{closed subsets}\right\}
  \quad\text{or}\quad
  \Phi:=\mc:=\left\{\text{compact subsets}\right\}.
\end{equation}
When $X$ is locally compact Hausdorff the latter is \emph{paracompactifying} in the sense of the same \cite[Definition I.6.1]{bred_shf_2e_1997} (also \cite[pre Proposition 3.3.2]{zbMATH03192990}): every member of $\Phi$ is \emph{paracompact} \cite[\S 5.1]{eng_top_1989} and has a closed neighborhood again belonging to $\Phi$. 

\begin{notation}\label{not:trngl}
  We denote actions both by juxtaposition and, occasionally and for emphasis, by the symbols `$\triangleright$' and `$\triangleleft$' for left and right actions respectively. 
\end{notation}

We take it for granted that quotients $X/\bG$ by proper actions $X\times \bG\xrightarrow{\alpha}X$ are always Hausdorff \cite[\S III.4.2, Proposition 3]{bourb_top_en_1} and locally compact Hausdorff when the original space is so \cite[\S III.4.5, Proposition 11]{bourb_top_en_1}. Locally compact groups $\bG$ are assumed Hausdorff unless specified otherwise. 

\begin{remarks}\label{res:prop.act}
  \begin{enumerate}[(1),wide]
  \item\label{item:res:prop.act:perp.not} For locally compact $\bG$ and Hausdorff $X$ \cite[\S III.4.4, Proposition 7]{bourb_top_en_1} equates the properness of the action $X\times \bG\xrightarrow{\alpha}X$ with the condition
    \begin{equation}\label{eq:prop.act}
      \left(\forall x,x'\in X\right)
      \left(\exists\text{ nbhds }V_x\ni x,\ V_{x'}\ni x'\right)
      \quad:\quad
      \braket{V_{x'}:V_x}_{\alpha}
      \subseteq \bG
      \text{ is relatively compact},
    \end{equation}
    where
    \begin{equation}\label{eq:a.b}
      \braket{A:B}_{\alpha}
      :=
      \left\{g\in \bG\ :\ Bg\cap A\ne\emptyset\right\}
      \quad
      \left(\text{the $((B,A))$ of \cite[\S 0]{palais_slice-nc}}\right).
    \end{equation}
    Following \cite[Definition 1.1.1]{palais_slice-nc}, we term the relative compactness of $\braket{A:B}_{\alpha}$ \emph{mutual (or relative) thinness} for $A$ and $B$, written $A\perp B$ (or $A\perp^{\alpha}B$, when wishing to highlight the action).

  \item\label{item:res:prop.act:bbt} For locally compact Hausdorff groups acting on \emph{Tychonoff} (or \emph{completely regular}, or $T_{3 \frac 12}$ \cite[Definition 14.8]{wil_top}) spaces, Palais gives a formally stronger notion of action properness in \cite[Definition 1.2.2]{palais_slice-nc} (note the alteration in quantifier structure, as compared to \Cref{eq:prop.act}):
    \begin{equation}\label{eq:pal.prop}
      \left(\forall x\in X\right)
      \left(\exists\text{ nbhd }V_x\ni x\right)
      \left(\forall x'\in X\right)
      \left(\exists\text{ nbhd }V_{x'}\ni x'\right)
      \quad:\quad
      V_x\perp  V_{x'}.
    \end{equation}
    The two notions coincide \cite[Theorem 1.2.9]{palais_slice-nc} when the carrier space $X$ is locally compact, but not in general (\cite[p.79, Example preceding \S 2]{MR271925},  \cite[Example 2.11.13 and Figure 2.11.14]{bs_stab_2002} or \cite[\S 2.4, Example 3]{MR61821}, following Bebutov). Consequently, following common practice in resolving the ambiguity (e.g. \cite[\S 1]{zbMATH01992380}, \cite[\S 10.5]{zbMATH05904318}), we refer to \Cref{eq:prop.act} and \Cref{eq:pal.prop} as \emph{Bourbaki} and \emph{Palais properness} respectively.

    We occasionally revert back to plain \emph{proper} when the distinction is irrelevant (e.g. when $X$ is locally compact, as mentioned).     
    
  \item\label{item:res:prop.act:princ.g.bun} Given freeness, Bourbaki properness for an action of a discrete group on a Hausdorff space also ensures that the action is \emph{properly discontinuous} (\cite[Definition preceding Theorem 81.5]{mnk}, \cite[post Lemma V.8.1]{massey_basic-alg-top_1991}, \cite[post Remark 3.1.6]{td_alg-top}):
    \begin{equation}\label{eq:prop.disc}
      \left(\forall x\in X\right)
      \left(\exists\text{ nbhd }U\ni x\right)
      \left(\forall 1\ne g\in \bG\right)
      \quad:\quad
      Ug\cap U=\emptyset. 
    \end{equation}
    \Cref{eq:prop.disc} is equivalent to $X\xrightarrowdbl{}X/\bG$ being a \emph{locally trivial} \cite[Definition 2.6.2]{hus_fib} \emph{principal $\bG$-bundle} \cite[Definition 4.2.2]{hus_fib} (cf. also \cite[\S 14.1]{td_alg-top}, where principal bundles are by definition locally trivial).

    As to the implication claim, it is essentially \cite[Theorem 1.2.9, (1) $\Rightarrow$ (2)]{palais_slice-nc} (which does not make use of that statement's local compactness assumption on the space $X$). 
    
  \item Conversely, a free action of a discrete group on a Hausdorff space will be Bourbaki proper if it is properly discontinuous and the quotient $X/\bG$ (equipped with the \emph{quotient topology} \cite[\S I.2.4, Example (1)]{bourb_top_en_1}) is Hausdorff: if $X/\bG$ is Hausdorff for a properly discontinuous action, one can take for the $V_{x,x'}$ of \Cref{eq:prop.act} preimages of disjoint neighborhoods of
    \begin{equation*}
      \pi(x)\text{ and }\pi(x')
      ,\quad
      X\xrightarrowdbl[\text{the obvious quotient map}]{\quad\pi\quad}X/\bG
    \end{equation*}
    The separation of $X/\bG$ is not automatic, given proper discontinuity: \cite[paragraph post Proposition 1.1.4]{palais_slice-nc} gives an example of an $\bR$-action with non-Hausdorff quotient which restricted to $\bZ\le \bR$ will be properly discontinuous with non-Hausdorff quotient. The back-and-forth passage between $\bR$ and its \emph{cocompact} subgroup $\bZ$ is examined more broadly in \Cref{le:coco} below.
  \end{enumerate}
\end{remarks}


\Cref{th:free.prop.zd.cpct} concerns proper $\bZ^d$-actions and the resulting cohomology on compactly-supported section spaces (that $H^p=0$ for $p>d$ follows from the fact \cite[Corollary 10.58]{rot} that $\bZ^d$ has \emph{cohomological dimension} $d$). For $\cF\in \cA b_X^{\bG}$ and $X\xrightarrow{\pi}X/\bG$ It will be convenient to use the notation
\begin{equation*}
  \cF/\bG
  :=
  (\pi_*\cF)^{\bG}
  \quad
  \left(\text{the \emph{sheaf of invariants} \cite[post Proposition 5.1.2]{zbMATH03192990}}\right):
\end{equation*}
its sections over $U\subseteq X/\bG$ are the $\bG$-invariant sections over $\pi^{-1}(U)\subseteq X$. Recall also \cite[Definition II.9.1]{bred_shf_2e_1997} that \emph{$\Phi$-soft} sheaves are those whose sections over $\Phi$-members extend globally (for paracompactifying $\Phi$, the case of exclusive interest below, this is equivalent by \cite[Proposition II.9.3]{bred_shf_2e_1997} to the alternative definition for \emph{$\Phi$-mou} of \cite[\S 3.5]{god_faisc_1958}). We write
\begin{equation}\label{eq:supp.fam}
  \Supp(\Theta)
  :=
  \left\{\text{family of supports generated by the closed-set family $\Theta$}\right\}.
\end{equation}




\begin{theorem}\label{th:free.prop.zd.cpct}
  For a proper action $X\times \bZ^d\to X$ on a locally compact Hausdorff space and a $\Supp\left(\mc\bZ^d\right)$-soft sheaf $\cF\in \cA b_X^{\bZ^d}$ for
  \begin{equation*}
    \mc\bZ^d
    :=
    \left\{K\bZ^d\ :\ K\in \mc:=\text{compact subsets of }X\right\}
  \end{equation*}
  we have
  \begin{equation*}
    H^p\left(\bZ^d,\ \Gamma_{\mc}(\cF)\right)
    \cong
    \begin{cases}
      \Gamma_{\mc}\left(X/\bZ^d,\ \cF/\bZ^d\right)
      &\text{if }p=d\\
      0
      &\text{otherwise}.
    \end{cases}    
  \end{equation*}
\end{theorem}

\begin{remark}\label{re:czd.parac}
  The statement's support family $\Supp\left(\mc\bZ^d\right)$ is paracompactifying: the properness assumption ensures \cite[Proposition 1.4(b)]{zbMATH03603988} that the restricted action map
  \begin{equation*}
    K\times \bZ^d\xrightarrow{\quad}X
    ,\quad
    K\subseteq X\text{ compact}
  \end{equation*}
  is proper so in particular closed, its domain $K\times \bZ^d$ is of course paracompact, so its image must be so too by the stability of paracompactness under closed maps \cite[Theorem 5.1.33]{eng_top_1989}. See also \Cref{cor:sweep.paraco} for more in the same vein. 
\end{remark}

\begin{proof}
    The proof inducts on $d\in \bZ_{\ge 0}$, with the base case $d=0$ tautological.

  \begin{enumerate}[(I),wide]
  \item\textbf{: the induction step.} This will be an application of the \emph{Lyndon-Hochschild-Serre spectral sequence} \cite[\S 6.8.2]{weib_halg}
    \begin{equation*}
      E_2^{p,q}
      :=
      H^p\left(
        \bZ
        ,\
        H^q\left(
          \bZ^{d-1},\ \bullet
        \right)
      \right)
      \xRightarrow{\quad}
      H^{p+q}
      \left(
        \bZ^d,\ \bullet
      \right).
    \end{equation*}
    The induction hypothesis collapses the spectral sequence on the second page, yielding
    \begin{equation*}
      \begin{aligned}
        H^p\left(\bZ^d,\ \Gamma_{\mc}(\cF)\right)
        &\cong
          E_2^{1,p-1}
          =
          H^1\left(
          \bZ
          ,\
          H^{p-1}
          \left(
          \bZ^{d-1},\ \Gamma_{\mc}(\cF)
          \right)
          \right)\\
        &\cong
          \begin{cases}
            H^1\left(\bZ,\ \Gamma_{\mc}\left(X/\bZ^{d-1},\ \cF/\bZ^{d-1}\right)\right)
            &p=d\\
            0
            &\text{otherwise}
          \end{cases}
          \quad\left(\text{induction hypothesis for }d-1\right)\\
        &\cong
          \begin{cases}
            \Gamma_{\mc}\left(X/\bZ^{d},\ \cF/\bZ^d\right)
            &p=d\\
            0
            &\text{otherwise}
          \end{cases}
          \quad\left(\text{induction hypothesis for }d=1\right).\\
      \end{aligned}
    \end{equation*}

    
  \item\textbf{: general remarks on $d=1$.} Only the case $p=1$ is interesting: higher cohomology certainly vanishes because $\bZ$ has (as recalled above \cite[Corollary 10.58]{rot}) cohomological dimension $cd(\bZ)=1$, and, the action being proper, there no compactly-supported $\bZ$-invariant sections (i.e. $H^0$ vanishes as well).
    
    Writing $n\triangleright$ for the action of $n\in \bZ$ on a left $\bZ$-module $A$ (\Cref{not:trngl}), observe first that any space $Z^1(\bZ,A)$ of \emph{1-cocycles} \cite[Application 6.5.5]{weib_halg} is in natural bijection with $A$:
    \begin{equation}\label{eq:1cocyca}
      \bZ\ni n
      \xmapsto{\quad}
      \begin{cases}
        \displaystyle
        a_n:=\sum_{k=0}^{n-1} \left(k\triangleright a \right)
        &\text{if }n\ge 0\\
        a_n
        :=
        -\left(n\triangleright a_{-n}\right)
        =
        \displaystyle
        -\sum_{k=n}^{-1}\left(k\triangleright a\right)
        &\text{if }n< 0\\
      \end{cases}
    \end{equation}
    is the unique 1-cocycle $(a_n)_{n\in \bZ}$ with $a_1=a$.
    
    In the present setting we can thus identify $Z^1\left(\bZ,\Gamma_{\mc}\right)\cong \Gamma_{\mc}$, and the cohomology group is
    \begin{equation}\label{eq:ck.quot}
      H^1\left(\bZ,\Gamma_{\mc}\right)
      \cong
      \Gamma_{\mc}/\im\left(1\triangleright -\id\right)|_{\Gamma_{\mc}}.
    \end{equation}
    The desired identification with the space of $\bZ$-periodic $c$-supported sections will be (the factorization through the quotient \Cref{eq:ck.quot} of) the map
    \begin{equation}\label{eq:periodization}
      \Gamma_{\mc}(\cF)
      \ni
      s
      \xmapsto{\quad}      
      \begin{aligned}
        S
        &:=
          \left(
          X\ni x
          \xmapsto{\quad}
          \lim_{t\to\infty}
          \sum_{k\ge 1}
          (t-k)\triangleright s
          \right)
        \\
        &=
          \lim_{t\to\infty}
          \sum_{k\ge 1}
          s\left(\bullet\triangleleft(t-k)\right)\triangleleft(k-t)\\        
        &\in
          \Gamma_{\mc}\left(X/\bZ,\ \cF/\bZ\right):
      \end{aligned}
    \end{equation}
    \begin{itemize}[wide]
    \item the sum is pointwise finite by the compact-support assumption on $s$;

    \item and
      \begin{equation}\label{eq:ftilde}
        \widetilde{s}
        :=
        \sum_{k\ge 1}(-k)\triangleright s
        =
        \sum_{k\ge 1}s\left(\bullet\triangleleft(-k)\right)\triangleleft k
      \end{equation}
      has the property that $s=1\triangleright \widetilde{s}-\widetilde{s}$, so that $t\triangleright \widetilde{s}$ stabilizes for large $t\in \bZ_{\ge 0}$ and the limit $S$ is $\bZ$-stable.
    \end{itemize}

  \item\textbf{$d=1$: \Cref{eq:periodization} induces an injective map on \Cref{eq:ck.quot}.} That it does indeed induce such a map to begin with follows from the observation that whenever
    \begin{equation*}
      s=1\triangleright s'-s'
      ,\quad
      s'\in \Gamma_{\mc}(\cF)
    \end{equation*}
    we have
    \begin{equation*}
      \sum_{k\ge 1}
      (t-k)\triangleright s
      =
      \sum_{k\ge 1}
      (t+1-k)\triangleright s'
      -
      \sum_{k\ge 1}
      (t-k)\triangleright s',
    \end{equation*}
    vanishing at $t\to\infty$. 

    As for injectivity, simply observe that
    \begin{equation*}
      S=0
      \xRightarrow{\quad}
      \bigg(
      \text{\Cref{eq:ftilde} has compact support}
      \bigg)
      \xRightarrow{\quad}
      \bigg(
      s=1\triangleright \widetilde{s}-\widetilde{s}
      =0\text{ in \Cref{eq:ck.quot}}
      \bigg).
    \end{equation*}
    
  \item\textbf{$d=1$: \Cref{eq:periodization} is onto.} Regard an arbitrary compactly-supported
    \begin{equation}\label{eq:sect.downstairs}
      s'\in \Gamma_{\mc}\left(X/\bZ,\ \cF/\bZ\right)
    \end{equation}
    as a $\bZ$-invariant section of $\cF$ supported on $K\bZ$ for compact $K\subseteq X$. Consider next
    \begin{itemize}[wide]
    \item a compact neighborhood $K'\supseteq K$;

    \item and $M\in \bZ_{>0}$ sufficiently large to ensure that the (automatically closed, by action properness) subsets $K'\bZ_{\le -M}$ and $K'\bZ_{\ge M}$ of $X$ are disjoint.
    \end{itemize}
    Define
    \begin{equation*}
      \Gamma\left(\cF|_{K'\bZ_{\le -M} \cup K'\bZ_{\ge M}\cup \partial K'\bZ}\right)
      \ni
      \widetilde{s}
      :=
      \begin{cases}
        s'&\text{on $K'\bZ_{\ge M}$}\\
        0&\text{on $K'\bZ_{\le -M}$}\\
        0&\text{on $\partial K'\bZ$}
      \end{cases}
    \end{equation*}
    The softness assumption allows us to extend $\widetilde s$ globally, and we may as well assume it supported in $K'\bZ$ because by construction it vanishes on that set's boundary. Finally, the image through \Cref{eq:periodization} of $s:=1\triangleright \widetilde{s}-\widetilde{s}$ will be the pre-selected \Cref{eq:sect.downstairs}.
  \end{enumerate}
\end{proof}

In particular, the sheaf of (germs of) continuous functions on $X$ of course acquires a canonical equivariant structure for any action. So do the sheaves of sections of $C^k$, $0\le k\le \infty$ (i.e. $k$-fold continuously differentiable) \emph{natural bundles} \cite[\S 6.14, Examples]{kms_natop} on smooth manifolds: $C^k$ functions, differential forms, vector fields, etc. The following remark encompasses these examples.

\begin{corollary}\label{cor:mfld}
  Let $M\times \bZ^d\to M$ be a free and proper action on a smooth manifold, $\cE$ a smooth $\bZ^d$-equivariant vector bundle, and denote by $C^k_{\mc}(\bullet)$ the spaces of compactly-supported $C^k$ sections for $0\le k\le \infty$.

  We then have
  \begin{equation*}
    H^p\left(\bZ^d,\ C_{\mc}^k(\cE)\right)
    \cong
    \begin{cases}
      C_{\mc}^k(\cE/\bZ^d)
      &\text{if }p=d\\
      0
      &\text{otherwise}.
    \end{cases}
  \end{equation*}  
\end{corollary}
\begin{proof}
  The action being also free (as well as proper), $M/\bZ^d$ will again be a smooth manifold \cite[Theorem 21.13]{lee2013introduction} and $\cE/\bZ^d$ a smooth vector bundle thereon. 
  
  Sheaves of (\emph{germs} \cite[p.2]{bred_shf_2e_1997} of) $C^k$ sections are modules over the sheaf of smooth germs. The latter is $\Phi$-soft \cite[Example II.9.4]{bred_shf_2e_1997} for any paracompactifying support family $\Phi$ (such as $\Phi:=\Supp\left(\mc\bZ^d\right)$: \Cref{re:czd.parac}), along with its modules \cite[Theorem II.9.16, Example II.9.17]{bred_shf_2e_1997}.
\end{proof}

In particular, when $\cE$ is the trivial (real) line bundle one recovers cohomology valued in spaces of compactly-supported $C^k$ functions. We may as well drop the freeness requirement: by \cite[Example 11]{zbMATH05718348} or \cite[Proposition 13.2.1]{thurst_3mfld} the quotient $M/\bZ^d$ will in that case be an \emph{orbifold} in the sense of \cite[Definition 2.2]{MR2883692}, \cite[Definition 2.1]{zbMATH05546218}, \cite[\S 13.2]{thurst_3mfld}, \cite[\S 1]{MR1466622} etc.

There are several notions of $C^k$ morphism between smooth orbifolds in the literature, but when the codomain is $\bR$ the four flavors in \cite[\S 2]{zbMATH06223376}, for instance, all converge. $C^k$ \emph{functions} on a smooth orbifold are thus unambiguous, and \Cref{th:free.prop.zd.cpct} specializes to

\begin{corollary}\label{cor:orbif}
  For a free and proper action  $M\times \bZ^d\to M$ on a smooth manifold and $0\le k\le \infty$ we have 
  \begin{equation*}
    H^p\left(\bZ^d,\ C_{\mc}^k(M)\right)
    \cong
    \begin{cases}
      C_{\mc}^k(M/\bZ^d)
      &\text{if }p=d\\
      0
      &\text{otherwise}.
    \end{cases}
  \end{equation*}
  \qedhere
\end{corollary}



\begin{remark}\label{re:zd.vs.cpct.lie}
  Because in the context of \Cref{th:free.prop.zd.cpct}, assuming the action free, $X\xrightarrowdbl{\pi}X/(\bG:=\bZ^d)$ is a principal $\bG$-bundle and the mutually-inverse functors
  \begin{equation*}
    \begin{tikzpicture}[>=stealth,auto,baseline=(current  bounding  box.center)]
      \path[anchor=base] 
      (0,0) node (l) {$\cA b_X^{\bG}$}
      +(6,0) node (r) {$\cA b_{X/\bG}$}
      ;
      \draw[->] (l) to[bend left=10] node[pos=.5,auto] {$\scriptstyle \cF\xmapsto{\quad}\cF/\bG $} (r);
      \draw[->] (r) to[bend left=10] node[pos=.5,auto] {$\scriptstyle \pi^*\cG \reflectbox{\ensuremath{\xmapsto{\quad}}} \cG$} (l);
    \end{tikzpicture}
  \end{equation*}
  constitute an equivalence of categories (\cite[\S 0.3, Lemma]{bl_eq-shv} or \cite[\S 5.1, post (5.1.2)]{zbMATH03192990}, say). \Cref{th:free.prop.zd.cpct} could have thus been phrased with
  \begin{equation*}
    \left(\cG\in \cA b_{X/\bG}\text{ and }\pi^*\cG\right)
    \quad\text{in place of}\quad
    \left(\cF/\bG\in \cA b_{X/\bG}\text{ and }\cF\right).
  \end{equation*}
  In that form, the result identifies the cohomology $H^p(\bT^d,\ \Gamma_{\mc}(\pi^*\cG))$ on the \emph{classifying space} $\bT^d\cong B_{\bZ^d}$ (\cite[\S\S II.1 and II.3]{am_coh-gp}, \cite[\S 7]{hjjm_bdle}) of a sheaf attached to the $\bZ^d$-module $\Gamma_{\mc}(\pi^*\cG)$.

  Note the contrast between \Cref{th:free.prop.zd.cpct} and the analogous situation for \emph{compact} acting groups $\bG$: there is then similarly a spectral sequence \cite[Theorem IV.9.2]{bred_shf_2e_1997}
  \begin{equation*}
    E_2^{p,q}
    :=
    H^p\left(B_{\bG},\ \cH^q_{\mc}(X,\pi^*\cG)\right)
    \xRightarrow{\quad}
    H_{\mc}^{p+q}(X/\bG,\ \cG)
  \end{equation*}
  if
  \begin{itemize}
  \item $X$ is assumed locally compact Hausdorff;

  \item the action is free (and automatically proper \cite[\S II.4.1, Proposition 2(a)]{bourb_top_en_1}, the group being compact);

  \item and $\cH^q_{\mc}(X,\pi^*\cG)$ is a sheaf on $B_{\bG}$ with stalks $H^q_{\mc}(X,\pi^*\cG)$ (the compactly-supported cohomology of $\pi^*\cG$ on $X$ \cite[Definition II.2.2]{bred_shf_2e_1997}), defined in \cite[\S IV.9]{bred_shf_2e_1997} as an instance of the \emph{Leray sheaf} construction of \cite[Definition IV.4.1]{bred_shf_2e_1997}. 
  \end{itemize}
  In particular, when $\cG$ and $\pi^*\cG$ happen to be $\mc$-soft the higher cohomologies $H^q_{\mc}(X,\pi^*\cG)$ and $H^q(X/\bG,\cG)$, $q\ge 1$ vanish \cite[Theorem II.9.11]{bred_shf_2e_1997} and hence so does $H^p(B_{\bG},\ \Gamma_{\mc}(\pi^*\cG))$ for $p\ge 1$.

  This is precisely what happens in the context of \Cref{cor:mfld} (\cite[Example II.9.17]{bred_shf_2e_1997} notes this for $C^{\infty}$; the remark applies equally to all $0\le k\le \infty$). 
\end{remark}

\section{Families of supports more broadly}\label{se:fams}

To move away from compact supports, we henceforth assume (much as \cite[\S III.4.4, Proposition 7]{bourb_top_en_1} does) $\bG$ locally compact Hausdorff and $X$ at least Hausdorff.

\begin{remark}\label{re:cpct.clsd}
  In the sequel we repeatedly use the fact that for any action $X\times \bG\xrightarrow{\alpha}X$ of a Hausdorff topological group on a Hausdorff space 
  \begin{equation*}
    \left(F\subseteq X\text{ closed}\right)
    \ 
    \&
    \ 
    \left(K\subseteq \bG\text{ compact}\right)
    \xRightarrow{\quad}
    FK\text{ is closed}. 
  \end{equation*}
  This is a standard result (as sketched e.g. in \cite[\S 29, Supplementary Exercise 11]{mnk} for the usual translation action):
  \begin{itemize}[wide]
  \item consider a \emph{convergent net} \cite[Definition 11.2]{wil_top}
    \begin{equation*}
      y_{\lambda}k_{\lambda}
      \xrightarrow[\quad\lambda\quad]{\quad}
      y
      ,\quad
      y_{\lambda}\in F,\ k_{\lambda}\in K;
    \end{equation*}
  \item assume $(k_{\lambda})_{\lambda}$ convergent to $k\in K$ (possible by compactness \cite[Theorem 17.4]{wil_top}, perhaps after passing to a subnet);

  \item which then yields
    \begin{equation*}
      y_{\lambda} = \left(y_{\lambda}k_{\lambda}\right)k_{\lambda}^{-1}
      \xrightarrow[\quad\lambda\quad]{\quad}
      yk^{-1}
      \xRightarrow{\quad}
      y\in Fk\subseteq FK.
    \end{equation*}
  \end{itemize}
\end{remark}

The following notion transports \Cref{eq:prop.act} over to point-closed set (rather than point-point) pairs.


\begin{definition}\label{def:transv.2.supp}
  Let $X\times \bG\xrightarrow{\alpha}X$ be an action of a locally compact group on a Hausdorff space.

  \begin{enumerate}[(1),wide]
  \item\label{item:def:transv.2.supp:f} For a closed subset $F\subseteq X$, $\alpha$ is \emph{$F$-proper} if
    \begin{equation*}
      \left(\forall x\in X\right)
      \left(\exists\text{ nbhds }V_x\ni x,\ V_F\supset F\right)
      \quad:\quad
      V_F\perp V_x.
    \end{equation*}

  \item\label{item:def:transv.2.supp:phi} Similarly, $\alpha$ is \emph{$\Phi$-proper} for a family $\Phi\subseteq 2^X$ of closed subsets if it is $F$-proper for every $F\in \Phi$.
  \end{enumerate}
  Note that Bourbaki properness means precisely $x$-properness for all points $x\in X$.
\end{definition}

\begin{remark}\label{re:small}
  In the language of \cite[Definition 1.2.1]{palais_slice-nc}, $F$-properness means that $F$ has a \emph{small} neighborhood. It follows that $F$ itself is small:
  \begin{equation*}
    \left(\forall x\in X\right)
    \left(\exists\text{ nbhd }V_x\ni x\right)
    \quad:\quad
    F\perp V_x.
  \end{equation*}
  For the applications of interest below it will make little difference whether we work with small closed sets or closed sets with small neighborhoods, for those sets will range over paracompactifying families of supports. 
\end{remark}

\begin{examples}\label{exs:trnsvrs}
  \begin{enumerate}[(1),wide]
  \item Actions of compact groups are of course $\Phi$-proper for arbitrary $\Phi$.

  \item Bourbaki-proper actions on Hausdorff spaces are $\mc$-proper (recall \Cref{eq:cl.c} that $\mc$ denotes the family of compact subsets). See for instance \cite[Theorem 11, (1) $\Leftrightarrow$ (2)]{2301.05325v3}, which does not use the $1^{st}$ countability assumption on $X$ made there crucially. 

  \item\label{item:exs:trnsvrs:lines} The horizontal flow
    \begin{equation*}
      \bR^2\times \bR
      \ni
      ((x,y),t)
      \xmapsto{\quad}
      (x+t,y)
      \in \bR^2
    \end{equation*}
    is $\Phi$-proper for $\Phi:=\left\{\text{vertical lines}\right\}$.

  \item\label{item:exs:trnsvrs:inv.prop} An action is $F$-proper for some non-empty \emph{invariant} $F$ precisely when the operating group is compact.
  \end{enumerate}
\end{examples}

The simple observation in \Cref{le:coco} will produce new $F$-proper actions out of old. Recall \cite[Proposition B.2.2]{bdv} that closed subgroup $\bH\le \bG$ of a locally compact Hausdorff group is \emph{cocompact} if $\bG/\bH$ is compact (\cite[Definition 1.10]{ragh} calls such subgroups \emph{uniform}).

\begin{lemma}\label{le:coco}
  Let $X\times \bG\xrightarrow{\alpha}X$ be an action of a locally compact Hausdorff group on a Hausdorff space.

  If $\bH\le \bG$ is a closed cocompact subgroup, $\alpha$ is $F$-proper for closed $F\subseteq X$ (or Bourbaki/Palais proper) if and only if its restriction $\alpha^{\bH}$ to $\bH$ is respectively so.
\end{lemma}
\begin{proof}
  That the properties transport from $\alpha$ to $\alpha^{\bH}$ is immediate, given that 
  \begin{equation*}
    \braket{A:B}_{\alpha^{\bH}}
    \subseteq
    \braket{A:B}_{\alpha}
    ,\quad
    \forall A,B\subseteq X.
  \end{equation*}
  For the converse we focus on $F$-properness, the other claims being analogous. Consider a compact $K=K^{-1}\subseteq \bG$ with $\bH K=\bG$: pick a compact neighborhood $V_g\ni g$ for every $g\in \bG$, cover $\bG/\bH$ with finitely images of the $V_g$, and take for $K$ the union of those $V_g$ (symmetrizing to $K\rightsquigarrow K\cup K^{-1}$ if necessary).

  For neighborhoods $V_x\ni x$ and $V_F\supset F$ we have 
  \begin{equation*}
    \braket{V_F:V_x}_{\alpha}
    \subseteq
    \braket{V_F:V_xK}_{\alpha^{\bH}}.
  \end{equation*}
  It will thus suffice to show that neighborhoods $V_F$ and $V_x$ can be chosen so as to make the latter set $\braket{V_F:V_xK}_{\alpha^{\bH}}$ relatively compact (i.e. $V_F\perp^{\alpha^{\bH}} V_xK$, in the notation of \Cref{res:prop.act}\Cref{item:res:prop.act:perp.not}). To that end, cover the compact set $xK$ with finitely many open $V_i$, $1\le i\le n$ so that
  \begin{equation*}
    \forall 1\le i\le n
    \quad:\quad
    W_i\perp^{\alpha^{\bH}}V_i
    \quad\text{for some nbhd }W_i\supset F
  \end{equation*}
  (possible by the assumed $F$-properness of $\alpha^{\bH}$). We then have
  \begin{equation*}
    \bigcap_{i=1}^n W_i
    \quad
    \perp^{\alpha^{\bH}}
    \quad
    \bigcup_{i=1}^n V_i.
  \end{equation*}
  Set $V_F:=\bigcap_i W_i$, and take for $V_x$ any neighborhood of $x$ for which $V_xK\subseteq \bigcup_i V_i$:
  \begin{equation*}
    V_x:=X\setminus \left(X\setminus \bigcup_{i=1}^n V_i\right)K
  \end{equation*}
  will do for instance, since $x$ does not belong to the (closed: \Cref{re:cpct.clsd}) set $\left(X\setminus \bigcup_i V_i\right)K$.
\end{proof}

We revisit, briefly, the contrast between Bourbaki and Palais properness recalled in \Cref{res:prop.act}\Cref{item:res:prop.act:bbt}, cocompact subgroups being somewhat topical to the proof of the following extension of Bebutov's example on \cite[p.79]{MR271925}. For locally compact groups $\bH$ and $\bG$ we write $\bH\preceq \bG$ to indicate that $\bH$ embeds as a closed subgroup into $\bG$.


\begin{proposition}\label{pr:bbt.ext}
  Consider the following conditions on a locally compact group $\bG$.
  \begin{enumerate}[(a),wide]
  \item\label{item:pr:bbt.ext:bp} All Bourbaki-proper $\bG$-actions on Tychonoff spaces are also Palais-proper.

  \item\label{item:pr:bbt.ext:bp.free} Condition \Cref{item:pr:bbt.ext:bp} holds for \emph{free} actions.

  \item\label{item:pr:bbt.ext:bp.free.prc} Condition \Cref{item:pr:bbt.ext:bp.free} holds for free actions on paracompact spaces. 
    
  \item\label{item:pr:bbt.ext:zg} $\bG$ admits no closed embeddings $\bZ\le \bG$ (notation: $\bZ\not\preceq \bG$).

  \item\label{item:pr:bbt.ext:cpctop} $\bG$ is a union of compact open subgroups. 
  \end{enumerate}
  The following implications hold:
  \begin{equation*}
    \text{\Cref{item:pr:bbt.ext:bp}}
    \xRightarrow{\quad}
    \text{\Cref{item:pr:bbt.ext:bp.free}}
    \xRightarrow{\quad}
    \text{\Cref{item:pr:bbt.ext:bp.free.prc}}
    \xRightarrow{\quad}
    \text{\Cref{item:pr:bbt.ext:zg}}
    \xLeftrightarrow{\quad}
    \text{\Cref{item:pr:bbt.ext:cpctop}}.
  \end{equation*}
\end{proposition}
\begin{proof}
  \begin{enumerate}[label={},wide]
  \item\textbf{\Cref{item:pr:bbt.ext:bp}
      $\Rightarrow$
      \Cref{item:pr:bbt.ext:bp.free}
      $\Rightarrow$
      \Cref{item:pr:bbt.ext:bp.free.prc}} are formal.
    
  \item\textbf{\Cref{item:pr:bbt.ext:bp.free.prc} $\Rightarrow$ \Cref{item:pr:bbt.ext:zg}:} The already-cited \cite[Example 2.11.13]{bs_stab_2002} is of a Bourbaki- but not Palais-proper action by $\bR$, hence one by the cocompact subgroup $\bZ\le \bR$ by \Cref{le:coco}. In that example the space being acted upon is metric (an open subset of $\bR^2$), hence \cite[Theorem 20.9]{wil_top} paracompact.

    Consider, now, an action $Y\times \bH\xrightarrow{\alpha^{\bH}} Y$ with the requisite properties (free, Bourbaki-proper, not Palais-proper, with paracompact carrying space $Y$) and a closed embedding $\bH\le \bG$. The corresponding \emph{twisted product} \cite[\S II.2]{bred_cpct-transf}
    \begin{equation}\label{eq:tw.prod}
      X:=Y\times_{\bH}\bG
      :=
      Y\times \bG\bigg/\left((y,g)\sim \left(yh,h^{-1}g\right),\ h\in \bH\right).
    \end{equation}
    carries an obvious right $\bG$-action $\alpha^{\bG}$. That freeness transports over from $\alpha^{\bH}$ to $\alpha^{\bG}$ is immediate: the $\alpha^{\bG}$-isotropy group
    \begin{equation*}
      \bG_{[y,g]}
      :=
      \left\{g'\in \bG\ :\ [y,g]g'=[y,g]\right\}
    \end{equation*}
    of the class $[y,g]\in X=Y\times_{\bH}\bG$ of $(y,g)\in Y\times \bG$ is $g^{-1}\bH_y g$, trivial if $\alpha^{\bH}$ is free. It is not difficult to verify that Bourbaki properness carries over from $\alpha^{\bH}$ to $\alpha^{\bG}$: one approach would be to first prove Bourbaki properness for the $(\bH\times \bG)$-action on $Y\times \bG$ whereby $\bH$ acts as in \Cref{eq:tw.prod} and $\bG$ by translation on the right-hand factor, and then apply \Cref{le:pass2quot.prpr} below.

    That Palais properness fails for $\alpha^{\bG}$ if it does for $\alpha^{\bH}$ is even easier to see: a Palais-proper $\bG$-action would restrict to one by $\bH$ on the $\bH$-invariant subspace $Y\subseteq X$, and that restricted action is nothing but the original $\alpha^{\bH}$.

    Finally, the assumed paracompactness of $Y$ entails that of $X=Y\times_{\bH}\bG$ by \Cref{le:tw.prod.props} below.

  \item\textbf{\Cref{item:pr:bbt.ext:cpctop} $\Rightarrow$ \Cref{item:pr:bbt.ext:zg}} is obvious: arbitrary elements being contained in compact subgroups, none can generate closed discrete subgroups.

  \item\textbf{\Cref{item:pr:bbt.ext:zg} $\Rightarrow$ \Cref{item:pr:bbt.ext:cpctop}:}   Recall first (\cite[\S 4.6, Theorem]{mz} and \cite[Theorem 3.39]{hm_pro-lie-bk}) that every \emph{almost-connected} locally compact Hausdorff group (i.e. \cite[p.2, Definition 1]{hm_pro-lie-bk} having compact component group $\bH/\bH_0$) is a \emph{cofiltered limit} \spr{04AU}
  \begin{equation*}
    \bH
    \cong
    \varprojlim_{\substack{\bK\trianglelefteq \bH\\\bK\text{ compact}}}
    \left(\text{Lie quotient }\bH/\bK\right).
  \end{equation*}
  The quotient maps $\bH\xrightarrowdbl{}\bH/\bK$ are proper and hence closed, so torsion-free elements generating closed subgroups move up and down the surjection. Since almost-connected \emph{Lie} groups contain some closed $\bZ$ precisely when non-compact (in which case they contain a closed one-parameter subgroup $\bR$, etc.), this shows that for almost-connected $\bH$ the property $\bZ\not\preceq \bH$ is in fact equivalent to compactness.

  Applying this to $\bH:=\bG_0$ we can assume the identity component compact. The properness of $\bG\xrightarrowdbl{}\bG/\bG_0$ then equates the $\bZ\not\preceq$ property for the two groups, so that the substitution $\bG \rightsquigarrow \bG/\bG_0$ reduces the problem to \emph{totally disconnected} locally compact groups (i.e. \cite[Definition 29.1]{wil_top} one whose connected components are points).

  The conclusion will now follow from the theory of {\it scale functions}
  \begin{equation*}
    \bG
    \xrightarrow[\quad\text{power-preserving: $\mathsf{s}(x^n)=\mathsf{s}(x)^n$}\quad]{\quad\text{continuous }\mathsf{s}\quad}
    \left(\bZ_{>0},\cdot\right)
  \end{equation*}
  developed in \cite{MR1299067} (also \cite[\S\S 2 and 4]{zbMATH00797049}) for totally disconnected locally compact groups. The continuity and power preservation mean that whenever $\bZ\not\preceq \bG$ the scale function must be trivial, whereupon \cite[\S 2, item I on p.90]{zbMATH00797049} every element $g\in \bG$ normalizes some compact open subgroup $\bK_g\le \bG$. But then an arbitrary $g\in \bG$ belongs to $\bK_g\rtimes\braket{g}$, which semidirect product is compact (and of course open): $g$ by assumption does not generate a closed copy of $\bZ$, so some power $g^n$, $n\in \bZ_{>0}$ must belong to the neighborhood $\bK_g\ni 1\in \bG$.
  \end{enumerate}
\end{proof}

We record the following simple observation on property permanence under quotients. 

\begin{lemma}\label{le:tw.prod.props}
  Let $\bH\le \bG$ be a closed subgroup of a locally compact group, and $Y\times \bH\to Y$ an action on a Hausdorff space.
  
  If $Y$ is (completely) regular or paracompact then so, respectively, is the twisted product $Y\times_{\bH}\bG$.
\end{lemma}
\begin{proof}
  Observe first that (complete) regularity and paracompactness all lift from $Y$ to the product $Y\times \bG$: products of (completely) regular spaces are again such \cite[Theorem 2.3.11]{eng_top_1989}, and a product of a paracompact space with a locally compact paracompact one (such as $\bG$ \cite[\S III.4.6, Proposition 13]{bourb_top_en_1}) is paracompact \cite[Proposition 1.11]{zbMATH03603988}.

  \begin{enumerate}[(I),wide]
  \item \textbf{: (Complete) regularity.} Observe that the Palais properness \cite[Theorem 1.1]{zbMATH05808276} of the translation action of $\bH$ on $\bG$ ensures (\cite[post Definition 1.3]{zbMATH03603988} or \cite[Proposition 1.3.3]{palais_slice-nc}) the same for the action
    \begin{equation}\label{eq:diag.h.act}
      Y\times \bG
      \ni
      (y,\ g)
      \xmapsto{\quad\triangleleft h\quad}
      (yh,\ h^{-1}g)
      ,\quad
      h\in \bH
    \end{equation}
    of \Cref{eq:tw.prod}. That the quotient $Y\times_{\bH}\bG\cong \left(Y\times \bG\right)/\bH$ is completely regular if $Y\times \bG$ is follows from \cite[Proposition 1.2.8]{palais_slice-nc}.

    As to \emph{plain} regularity, we isolate its permanence under quotients by Palais-proper actions in \Cref{le:reg.perm} below. 
    
  \item \textbf{: Paracompactness.} \cite[Theorem 1.2]{zbMATH05808276} ensures the existence of a closed \emph{$\bH$}-fundamental subset $F\subseteq \bG$: small for the translation $\bH$-action, such that $\bH F=\bG$. The product
    \begin{equation*}
      Y\times F
      \subseteq
      Y\times \bG
    \end{equation*}
    is then closed and $\bH$-fundamental for \Cref{eq:diag.h.act} and hence \cite[Proposition 1.4 (a) $\Rightarrow$ (b)]{zbMATH03603988} the restriction the quotient map $Y\times \bG\xrightarrowdbl{} Y\times_{\bH}\bG$ is closed. The conclusion now follows from the fact \cite[5.1.33]{eng_top_1989} that paracompactness transports over to Hausdorff images by closed continuous maps.
  \end{enumerate}
\end{proof}

The already-invoked \cite[Proposition 1.2.8]{palais_slice-nc}, proving $T_{3\frac 12}$ permanence under Palais-proper quotients, will not quite yield \Cref{le:reg.perm} below (appealed to in the proof of \Cref{le:tw.prod.props}): the former's proof used complete regularity crucially. It is a simple matter, however, to prove the analogous permanence result for ordinary regularity. 

\begin{lemma}\label{le:reg.perm}
  The quotient $X/\bG$ of a regular space through a Palais-proper action $X\times \bG \to X$ of a locally compact Hausdorff group is again regular. 
\end{lemma}
\begin{proof}
  The claim amounts to this: if the orbit $x\bG$ of $x\in X$ does not intersect the $\bG$-invariant closed set $F=\overline{F}\subseteq X$, then $x$ has a closed $\bG$-invariant neighborhood that also fails to meet $F$.  

  By the properness assumption $x$ has a small neighborhood $V_x$ in the sense of \cite[Definition 1.2.1]{palais_slice-nc}:
  \begin{equation*}
    \left(\forall y\in X\right)
    \left(\exists\text{ nbhd }W\ni y\right)
    \quad:\quad
    W\perp V_x.
  \end{equation*}
  Furthermore, by regularity we may as well assume
  \begin{equation*}
    V_x=\overline{V_x}
    \quad\text{and}\quad
    V_x\cap F=\emptyset.
  \end{equation*}
  I now claim that
  \begin{equation}\label{eq:good.y.nbhds}
    \left(\forall y\in F\right)
    \left(\exists\text{ nbhd }W_y\ni y\right)
    \quad:\quad
    W_y\bG\cap V_x=\emptyset. 
  \end{equation}
  This would clinch the proof: $V_x\bG$ would be the desired $\bG$-invariant (closed: \cite[Proposition 1.4, (a) $\Rightarrow$ (b)]{zbMATH03603988}) neighborhood of $x$ intersecting $F$ trivially.

  To verify \Cref{eq:good.y.nbhds}, first employ smallness to conclude that for some neighborhood $W\ni y$ we have
  \begin{equation}\label{eq:wgvx}
    Wg\cap V_x=\emptyset
    ,\quad
    \forall g\in \bG\setminus K
    ,\quad
    K\subseteq \bG
    \text{ compact}.
  \end{equation}
  Because $yK$ and $V_x$ are closed disjoint subsets of a regular space with the former compact, they have \cite[Theorem 3.1.6]{eng_top_1989} disjoint open neighborhoods
  \begin{equation*}
    yK\subset U_y
    ,\quad
    V_x\subset U_x
    ,\quad U_y\cap U_x=\emptyset.
  \end{equation*}
  The set $\left(X\setminus U_y\right)K^{-1}\not\ni y$ being closed by \Cref{re:cpct.clsd},
  \begin{equation*}
    W_y
    :=
    \text{ the interior }
    \left(W\cap \bigcap_{k\in K} U_y k^{-1}\right)^{\circ}
  \end{equation*}
  contains $y$ and we have
  \begin{equation*}
    W_yk\subset U_y
    \xRightarrow{\quad}
    W_yk\cap V_x=\emptyset
    ,\quad
    \forall k\in K.
  \end{equation*}
  In conjunction with \Cref{eq:wgvx} (and $W_y\subseteq W$), this delivers \Cref{eq:good.y.nbhds}.
\end{proof}

\begin{remarks}\label{res:when.gh.bund}
  \begin{enumerate}[(1),wide]
  \item There are alternative approaches to the paracompactness of $Y\times_{\bH}\bG$, not relying on the crucial \cite[Theorem 1.2]{zbMATH05808276}, when $\bG\xrightarrowdbl{\quad} \bG/\bH$ is a locally trivial principal $\bH$-bundle (e.g. when $\bH$ is Lie \cite[Theorem 3]{most_sect}, which case it would have sufficed to consider in the proof of \Cref{pr:bbt.ext}: in the implication \Cref{item:pr:bbt.ext:bp.free.prc} $\Rightarrow$ \Cref{item:pr:bbt.ext:zg} we might as well have taken $\bH:=\bZ$):
    \begin{itemize}[wide]
      
    \item \cite[Theorem II.2.4]{bred_cpct-transf} realizes the twisted product $X$ as a locally trivial bundle $X\xrightarrowdbl{\pi}\bG/\bH$ with fiber $Y$ (assumed paracompact).
      
    \item Cover $\bG/\bH$ with the interiors of compact sets $K_i$ trivializing said bundle in the sense that
      \begin{equation*}
        \left(\pi^{-1}(K_i)\xrightarrowdbl{\pi}K_i\right)
        \quad\cong\quad
        \left(Y\times K_i\xrightarrowdbl[]{\quad\text{projection}\quad}K_i\right).
      \end{equation*}
      
    \item $\bG/\bH$ being paracompact \cite[\S III.4.6, Proposition 13]{bourb_top_en_1}, we may as well assume the family $\{K_i\}_i$ \emph{locally finite} \cite[preceding Theorem 1.1.11]{eng_top_1989}: every point of $\bG/\bH$ has a neighborhood intersecting only finitely many $K_i$.
      
    \item But then $\left\{Y\times K_i\right\}_i$ is a locally finite closed cover of $X$ by paracompact \cite[Theorem 5.1.36]{eng_top_1989} spaces, so $X$ must be paracompact \cite[p.59, post ($\Sigma$)]{zbMATH03289133}.
    \end{itemize}
    The argument is very much along the lines of that showing that products of paracompact and locally compact paracompact spaces are again paracompact (\cite[Proposition 1.11]{zbMATH03603988}, say).

  \item By contrast to (complete) regularity or paracompactness, \emph{normality} certainly need not be inherited from $Y$ by $Y\times_{\bH}\bG$. Assuming $\bH$ trivial to simplify matters, $Y\times \bG$ will fail to be normal for \emph{any} non-discrete locally compact group $\bG$ if $Y$ is a \emph{Dowker space} (i.e. \cite[p.763]{MR776636} a normal space whose product with the unit interval is not normal; such exist \cite[Theorem 1.1]{zbMATH00936741}):
    \begin{itemize}[wide]
    \item one may as well restrict attention to an open subgroup of $\bG$, thus assuming the latter almost-connected (and still non-discrete);
      
    \item in which case the normality of $Y\times \bG$ would travel \cite[Theorem 1.5.20]{eng_top_1989} along closed (because proper) maps
      \begin{equation*}
        Y\times \bG
        \xrightarrow{\quad}
        Y\times \left(\bG/\bK\right)
        ,\quad
        \text{compact }\bK\trianglelefteq \bG;
      \end{equation*}

    \item so that the product of $Y$ with some infinite \emph{metrizable} $\bG/\bK$ would be normal;

    \item violating the Dowker assumption by \cite[Theorem 1.1, (iii) $\iff$ (iv)]{MR776636}.
    \end{itemize}
  \end{enumerate}  
\end{remarks}

The concepts introduced in \Cref{def:transv.2.supp} are amenable to the sort of passage to quotient groups visible in the proof of \Cref{th:free.prop.zd.cpct} (see also \cite[Proposition 1.3.2]{palais_slice-nc} and \cite[Lemma 2.3]{zbMATH03603988} for analogues of the following remark).

\begin{lemma}\label{le:pass2quot.prpr}
  Let $X\times \bG\xrightarrow{\alpha}X$ be a Bourbaki-proper action of a locally compact Hausdorff group on a Hausdorff topological space, $F\subseteq X$ a closed subset and $\bH\trianglelefteq \bG$ a closed normal subgroup. 

  If $\alpha$ is $F$-proper, then the induced $(\bG/\bH)$-action on $X/\bH$ is proper with respect to the image of $F$ through $X\xrightarrowdbl{}X/\bH$, and that image is closed.
\end{lemma}
\begin{proof}
  \begin{enumerate}[(I),wide]
  \item \textbf{: closure.} Substituting back $\bG$ for $\bH$, the claim amounts to showing that whenever $\alpha$ is $F$-proper the image $\pi(F)$ through $X\xrightarrowdbl{\pi}X/\bG$ is closed. This follows from \cite[Proposition 1.2(d)]{zbMATH03603988} applied to the family $\left\{Fg\right\}_{g\in \bG}$, given the smallness of $F$ (\Cref{re:small}).

    
  \item \textbf{: properness.} For $x\in X$ once more choose neighborhoods $V_x\ni x$ and $V_F\supset F$ per \Cref{def:transv.2.supp} and observe that, writing $\overline{\alpha}$ for the $(\bG/\bH)$-action on $X/\bH$ induced by $\alpha$ and $\bG\xrightarrowdbl{\pi^{\bG}}\bG/\bH$,
    \begin{equation*}
      \braket{\pi(V_F):\pi(V_x)}_{\overline{\alpha}}
      =\pi^{\bG}\left(
        \braket{V_F:V_x}_{\alpha}
      \right).
    \end{equation*}
    The conclusion follows from the fact that $\pi^{\bG}$ sends relatively compact subsets to such. 
  \end{enumerate}
\end{proof}

In \Cref{le:pass2quot.prpr}, Bourbaki properness is only there to ensure good quotient behavior. The closure of the $X/\bG$-image of $F$ is equivalent to the closure of $F\bG$, and \emph{that} requires no Bourbaki properness: the argument already given in fact yields

\begin{corollary}\label{cor:fg.closed}
  Let $X\times \bG\xrightarrow{\alpha}X$ be an action of a locally compact Hausdorff group on a Hausdorff topological space, $F$-proper for closed $F\subseteq X$.

  The subset $F\bG\subseteq X$ is closed.  \qedhere
\end{corollary}

This is entirely analogous to the closure of the orbits \cite[\S III.4.2, Proposition 4(d)]{bourb_top_en_1} of a Bourbaki-proper action. As a consequence, a closed paracompact $F\subseteq X$ sweeps out a closed paracompact set under an $F$-proper action. 


\begin{corollary}\label{cor:sweep.paraco}
  Let $X\times \bG\xrightarrow{\alpha}X$ be an action of a locally compact Hausdorff group on a Hausdorff topological space.

  If $\alpha$ is $F$-proper for closed paracompact $F\subseteq X$, the closed set $F\bG\subseteq X$ is paracompact.  
\end{corollary}
\begin{proof}
  Closure is already covered by \Cref{cor:fg.closed}. As for paracompactness, note first that $F\times \bG\to F\bG$ is a \emph{perfect map} in the sense of \cite[Problem 20G]{wil_top}: continuous closed surjection with compact point preimages (properness and hence closure follows from \cite[Proposition 1.4 (a) $\Rightarrow$ (c)]{zbMATH03603988}; the assumed Palais properness of that statement does not play a role in proving the implication referenced here).

  The paracompactness of $F\bG$ is thus equivalent by \cite[Problem 20G]{wil_top} to that of $F\times \bG$. The latter is paracompact by \cite[Proposition 1.11]{zbMATH03603988} or \cite[Problem 5.5.5(c)]{eng_dim}, being the product of a paracompact space $F$ and a locally compact paracompact (\cite[\S III.4.6, Proposition 13]{bourb_top_en_1}, \cite[\S 1, first paragraph]{zbMATH05182649}) space $\bG$.  \qedhere

\end{proof}

We observe next that paracompactness also behaves well in passing to quotients by $\Phi$-proper actions.

\begin{lemma}\label{le:still.paraco}
  Let $X\times \bG\xrightarrow{\alpha}X$ be a Bourbaki-proper action of a locally compact Hausdorff group on a Hausdorff topological space, $F$-proper for closed $F\subseteq X$.

  If $F$ is paracompact, so is its image through $X\xrightarrowdbl{\pi}X/\bG$.
\end{lemma}
\begin{proof}
  This follows from the fact \cite[Theorem 20.12(b)]{wil_top} that paracompactness travels along continuous closed surjections between Hausdorff spaces: the (co)restriction
  \begin{equation*}
    F
    \xrightarrowdbl{\quad\pi|_F\quad}
    \pi(F)
  \end{equation*}
  is closed (and has closed image $\pi(F)\subseteq X/\bG$) by \Cref{cor:fg.closed}, which implies that $F'\bG\subseteq X$ is closed for every closed $F'\subseteq F$.
\end{proof}

The following bit of notation will be useful in generalizing \Cref{th:free.prop.zd.cpct} beyond compact supports.

\begin{definition}\label{def:push.supp}
  For a continuous map $X\xrightarrow{f}Y$ and a family of supports $\Phi\subseteq 2^X$ the \emph{push-forward} $f_*\Phi\subseteq 2^Y$ is the family of supports generated by
  \begin{equation*}
    \left\{\overline{f(F)}\ :\ F\in \Phi\right\}.
  \end{equation*}
  We will mostly be concerned with push-forwards in conditions when $f(F)$ are automatically closed. 
\end{definition}

\begin{examples}\label{exs:push}
  \begin{enumerate}[(1),wide]
  \item\label{item:exs:push:prop} If $f$ is a proper surjection of Hausdorff spaces preimages of compact sets are compact \cite[\S I.10.2, Proposition 6]{bourb_top_en_1}, and hence $f_*\mc=\mc$.

  \item\label{item:exs:push:open} Similarly, $f_* \mc=\mc$ if $X\xrightarrow{f}Y$ is an \emph{open} surjection between locally compact Hausdorff spaces (e.g. quotient maps $X\xrightarrowdbl{}X/\bG$ for Bourbaki-proper actions on locally compact spaces \cite[\S III.4.5, Proposition 11]{bourb_top_en_1}).

    Indeed, on the one hand images of compact sets are compact in any case, while openness affords going back: cover a compact $K\subseteq Y$ with finitely many interiors of compact sets of the form
    \begin{equation*}
      f(Q)
      ,\quad
      Q\text{ an open neighborhood of some }x\in f^{-1}(K).
    \end{equation*}

  \item\label{item:exs:push:comp} As expected, if $X\xrightarrow{f}Y\xrightarrow{g}Z$ are composable maps and $\Phi\subseteq 2^X$ is support family, then
    \begin{equation*}
      g_*\left(f_*\Phi\right)
      =
      (g\circ f)_*\Phi.
    \end{equation*}
    Simply observe that a closed subset of $Z$ contains one of the sets 
    \begin{equation*}
      g(f(F))
      \subseteq
      g\left(\overline{f(F)}\right)
      \subseteq
      \overline{g\left(\overline{f(F)}\right)}
      =
      \overline{g\left(f(F)\right)}
    \end{equation*}
    if and only if it contains the others.
  \end{enumerate}
\end{examples}

In conjunction with \Cref{def:push.supp}, we employ

\begin{notation}\label{not:push.supp}
  For an appropriately well-behaved quotient $X\xrightarrow{\pi}X/\bG$ and a support family $\Phi\subset 2^X$ we occasionally write $\Phi/\bG:=\pi_*\Phi$.
\end{notation}

The preceding material builds up to the following natural extension of \Cref{th:free.prop.zd.cpct}. 

\begin{theorem}\label{th:free.prop.zd.gen}
  Consider a Bourbaki-proper action $X\times \bZ^d\to X$ on a Hausdorff space, $\Phi$-proper in the sense of \Cref{def:transv.2.supp}\Cref{item:def:transv.2.supp:phi} for a $\bZ^d$-invariant paracompactifying family of supports $\Phi\subseteq 2^X$.
  
  If $\cF\in \cA b_X^{\bZ^d}$ is $\Supp\left(\Phi \bZ^d\right)$-soft for
  \begin{equation*}
    \Phi\bZ^d
    :=
    \left\{F\bZ^d\ :\ F\in \Phi \right\}
  \end{equation*}
  then 
  \begin{equation*}
    H^p\left(\bZ^d,\ \Gamma_{\Phi}(\cF)\right)
    \cong
    \begin{cases}
      \Gamma_{\pi_*\Phi}\left(X/\bZ^d,\ \cF/\bZ^d\right)
      &\text{if }p=d\\
      0
      &\text{otherwise},
    \end{cases}    
  \end{equation*}
  where $X\xrightarrowdbl{\pi}X/\bZ^d$ is the quotient map and $\pi_*\Phi$ is the push-forward of \Cref{def:push.supp}. 
\end{theorem}
\begin{proof}
  The proof of the earlier \Cref{th:free.prop.zd.cpct} adapts fairly readily, given the various auxiliary results. The argument is again inductive on $d$, with the inductive step easily dispatched. Write
  \begin{equation*}
    X/\bZ^{k}
    \xrightarrowdbl{\quad\pi_{k,\ell}\quad}
    X/\bZ^{\ell}
    ,\quad
    0\le k\le \ell\le d
  \end{equation*}
  for the intermediate quotient maps. \Cref{exs:push}\Cref{item:exs:push:comp} ensures that these are all compatible with the push-forwards of $\Phi$, in the sense that 
  \begin{equation*}
    \pi_{0,\ell*}\Phi
    =:
    \Phi_{\ell}
    =
    \pi_{k,\ell*}\Phi_{k}
    ,\quad
    \forall 0\le k\le \ell\le d.
  \end{equation*}
  Moreover, the images $\pi_{k,\ell}(F)$, $F\in \Phi_k$ are closed by \Cref{le:pass2quot.prpr}, and the same result transports the properness hypotheses over to the action of $\bZ^{d-\ell}$ of $X/\bZ^{\ell}$, $0\le \ell\le d$. Finally, all $\Phi_{\ell}$ are again paracompactifying by \Cref{le:still.paraco}. The hypotheses thus descend along the projections $\pi_{k,\ell}$, so the claim does indeed reduce to $d=1$ as in the earlier proof.

  Assuming henceforth that $d=1$, we once more retrace the steps of the earlier proof, with modifications where appropriate. The only potentially non-vanishing cohomology is $H^1$ for, as before,
  \begin{itemize}
  \item higher cohomology groups because $cd(\bZ)=1$;

  \item and the properness of the action ensures that a $\bZ$-invariant $\Phi$-supported section vanishes by \Cref{exs:trnsvrs}\Cref{item:exs:trnsvrs:inv.prop}.
  \end{itemize}
  This reduces the problem to showing that
  \begin{equation}\label{eq:th:free.prop.zd.gen:h1}
    H^1\left(\bZ,\ \Gamma_{\Phi}(\cF)\right)
    \cong
    \Gamma_{\Phi/\bZ}\left(X/\bZ,\ \cF/\bZ\right),
  \end{equation}
  whence we can again proceed as before: the analogue 
  \begin{equation*}
    \Gamma_{\Phi}(\cF)
    \in
    s
    \xmapsto{\quad}      
    \begin{aligned}
      S
      &:=
        \left(
        X\ni x
        \xmapsto{\quad}
        \lim_{t\to\infty}
        \sum_{k\ge 1}
        (t-k)\triangleright s
        \right)
      \\
      &=
        \lim_{t\to\infty}
        \sum_{k\ge 1}
        s\left(\bullet\triangleleft(t-k)\right)\triangleleft(k-t)\\        
      &\in
        \Gamma_{\Phi/\bZ}\left(X/\bZ,\ \cF/\bZ\right)
    \end{aligned}
  \end{equation*}
  of \Cref{eq:periodization} induces a morphism
  \begin{equation*}
    H^1\left(\bZ,\Gamma_{\Phi}(\cF)\right)
    \cong
    \Gamma_{\Phi}(\cF)/\im\left(1\triangleright -\id\right)|_{\Gamma_{\Phi}(\cF)}
    \xrightarrow{\quad}
    \Gamma_{\Phi/\bZ}\left(X/\bZ,\ \cF/\bZ\right).
  \end{equation*}
  Neither injectivity nor surjectivity are substantially harder to check than in \Cref{th:free.prop.zd.cpct}. Focusing on the latter, simply observe that the argument employed previously goes through given that under the hypotheses ($\Phi$ is paracompactifying and the action is $\Phi$-proper) the family $\Supp\left(\Phi\bZ^d\right)$ is again paracompactifying by \Cref{cor:sweep.paraco}. 
\end{proof}

\begin{remarks}\label{res:soft.not.desc}
  \begin{enumerate}[(1),wide]
  \item For a support family $\Phi$ invariant under a $\bG$-action, $\Supp(\Phi\bG)$-softness is one version of what one might mean by a sheaf being ``equivariantly soft''. Another would be the requirement that the descent $\cF/\bG$ be $\Phi/\bG$-soft; some care is in order however: the latter condition requires the global extensibility of \emph{$\bG$-invariant} sections from $F\bG$, $F\in \Phi$ again $\bG$-invariantly, and does not follow from the former (at least for Bourbaki properness).

    This is already visible in Bebutov's example recalled in \Cref{res:prop.act}\Cref{item:res:prop.act:bbt}, modified by passing from $\bR$ to the cocompact $\bZ\le \bR$, as in the proof of \Cref{pr:bbt.ext}, \Cref{item:pr:bbt.ext:bp.free.prc} $\Rightarrow$ \Cref{item:pr:bbt.ext:zg}. Take for $\Phi$ the family $\mathrm{cl}$ of all closed sets, and for $\cF$ the sheaf of continuous $\bR$-valued germs. This is of course a soft (i.e. $\mathrm{cl}$-soft) sheaf on the metric space $X\subseteq \bR^2$ acted upon by $\bZ$ (by Tietze \cite[Theorem 15.8]{wil_top}, metric spaces being normal \cite[Example 5.3(c)]{wil_top}). Because $X/\bZ$ is not completely regular, continuous functions generally do not extend globally from closed subsets.

  \item The sheaf $\cF/\bZ$ of continuous germs on the preceding observation's space $X/\bZ$, despite being non-soft, is nevertheless \emph{acyclic} (i.e. \cite[\S II.4]{bred_shf_2e_1997} has vanishing cohomology in degrees 1 and higher):
    \begin{itemize}[wide]
    \item The $\bZ$-action on $X$ being free, we have a spectral sequence \cite[Proposition 5.2.3]{zbMATH03192990}
      \begin{equation*}
        E_2^{p,q}
        :=
        H^p(\bZ,\ H^q(X,\cF))
        \xRightarrow{\quad}
        H^{p+q}(X/\bZ,\cF/\bZ).
      \end{equation*}
    \item Continuous-function softness over $X$ then collapses this to
      \begin{equation}\label{eq:hpz.hpf}
        H^p\left(\bZ,\ \cat{Cont}(X\to \bR)\right)
        \xRightarrow{\quad}
        H^{p}(X/\bZ,\cF/\bZ).
      \end{equation}
    \item And finally, \cite[Lemma 2.4]{zbMATH00825443} argues that the left-hand group cohomology in \Cref{eq:hpz.hpf} vanishes for $p\ge 1$.
    \end{itemize}
    
  \item\label{item:res:soft.not.desc:bo} The aforementioned acyclicity result \cite[Lemma 2.4]{zbMATH00825443} refers specifically to a proper action on
    \begin{equation*}
      \bL/\bK
      ,\quad
      \bL\text{ semisimple Lie}
      ,\quad
      \bK\le \bL\text{ maximal compact}
    \end{equation*}
    of a discrete subgroup $\bG\le \bL$. The proof, however, effectively shows that $H^p(\bG,\ \Gamma(\cF))$ vanish for $p\ge 1$ whenever the discrete group $\bG$ acts properly on a metrizable, locally compact, $\sigma$-compact space $X$ and $\cF\in \cA b_X^{\bG}$ is a sheaf of modules over real-valued continuous germs (the original setup is ever so slightly different: there is a smooth structure on the space and the sheaf is one of modules over smooth germs, etc.).
    
    The proof makes essential use, in the stated framework, of the existence of a $\bG$-invariant \emph{partition of unity} \cite[Problem 20C]{wil_top}: functions $X\xrightarrow{\varphi_{g}}\bR$, $g\in \bG$ whose supports constitute a locally finite family, with
    \begin{equation*}
      \sum_{g\in \bG}\varphi_g\equiv 1
      \quad\text{and}\quad
      g\triangleright \varphi_{g'} = \varphi_{g'}(\bullet\triangleleft g)=\varphi_{gg'}. 
    \end{equation*}
    That such a gadget exists is not difficult to see:
    \begin{itemize}[wide]
    \item Select open $U\subseteq X$ with $\left\{Ug\right\}_{g}$ locally finite (as the proof of \cite[Lemma 2.4]{zbMATH00825443} indicates).

    \item Shrink that to an open $V$ with $\overline{V}\subseteq U$ so that $\left\{Vg\right\}_{g\in \bG}$ is again a cover. This point and the preceding one do go through: the assumptions on the action (properness plus metrizability and local and $\sigma$-compactness for $X$) ensure \cite[Theorem 1.1]{zbMATH05904318} that the quotient $X/\bG$ is paracompact, whence \cite[\S 1.4, Lemma 2]{ksz_trnsf} the existence of a closed subset $F\subseteq X$ with
      \begin{equation*}
        F\bG=X
        ,\quad
        \exists\text{ small \cite[Definition 1.2.1]{palais_slice-nc} neighborhood }
        U\supset F
      \end{equation*}
      One can then take for $V$ a neighborhood of $F$ whose closure is contained in $U$ (possible by local and $\sigma$-compactness, hence paracompactness, hence normality). 

    \item Select a $[0,1]$-valued continuous $\widetilde{\varphi}_1$ supported on $U$ with $\widetilde{\varphi}_1|_V\equiv 1$.

    \item Define $\widetilde{\varphi}_g:=g\triangleright \widetilde{\varphi}_{1}$.
      
    \item Finally, set
      \begin{equation}\label{eq:div.by.sum}
        \varphi_g
        :=
        \frac {\widetilde{\varphi}_g}{\sum_{g'\in \bG}\widetilde{\varphi}_{g'}}
        \quad
        \left(\text{locally finite denominator, positive everywhere}\right).
      \end{equation}
    \end{itemize}
  \end{enumerate}  
\end{remarks}

The division by a positive sum in \Cref{eq:div.by.sum} makes crucial use of the order structure of $\bR$: while partitions of unity make sense (and exist \cite[Exercise II.13]{bred_shf_2e_1997}) for arbitrary soft sheaves of rings, \emph{invariant} partitions of unity as discussed in \Cref{res:soft.not.desc}\Cref{item:res:soft.not.desc:bo} above do not exist in full generality:

\begin{example}\label{ex:pos.char}
  Consider the case of a finite group $\bG$ acting on the singleton $X:=\{*\}$ and trivially on an algebra $\cF$ (regarded here as a sheaf on $X$) over a field of characteristic $p$ dividing $|\bG|$. 
\end{example}

One feature distinguishing the simple \Cref{ex:pos.char} from the more pleasant setup of \cite[Lemma 2.4]{zbMATH00825443} and \Cref{res:soft.not.desc}\Cref{item:res:soft.not.desc:bo} is the cohomology of the isotropy groups $\bG_x$, $x\in X$ (i.e. $\bG$ itself when all actions in sight are trivial, as in \Cref{ex:pos.char}): if $p\ |\ |\bG|$ then infinitely many cohomology groups
\begin{equation*}
  H^i(\bG,\bF_p)
  ,\quad
  i\in \bZ_{>0}
  ,\quad
  \bF_p:=\text{field with $p$-elements, $\bG$-acted upon trivially}
\end{equation*}
are non-vanishing \cite[Proposition II.7.5]{am_coh-gp}.

The extension of \cite[Lemma 2.4]{zbMATH00825443} in \Cref{pr:bo.gen} confirms that (under suitable conditions which certainly obtain in that result) this cohomological obstruction is the only hurdle precluding the type of acyclicity noted there. Incidentally, the same result also offers an instructive contrast to \Cref{th:free.prop.zd.gen}, where smaller supports preclude cohomology vanishing.

\begin{proposition}\label{pr:bo.gen}
  Let $\Phi\subseteq 2^X$ be a paracompactifying family of supports on a Hausdorff space $X$, $\bG$-invariant under a Bourbaki-proper $\Phi$-proper action $X\times \bG\to X$ by a discrete group.

  If $\cF\in \cA b_X^{\bG}$ is $\Supp(\Phi\bG)$-soft, $\cF/\bG$ is $\Phi/\bG$-soft and
  \begin{equation}\label{eq:isotr.gp.coh.0}
    H^i(\bG_x,\ \cF_x)=0
    ,\quad
    \forall i>0
    ,\quad
    \forall x\in X
  \end{equation}
  then $H^i(\bG,\ \Gamma_{\Supp(\Phi\bG)}(\cF))$ vanish for $i>0$.
\end{proposition}
\begin{proof}
  Most of the argument is already sketched in \Cref{res:soft.not.desc}\Cref{item:res:soft.not.desc:bo}. Bourbaki properness implies Grothendieck's \cite[\S 5.3, Condition (D)]{zbMATH03192990}: isotropy groups $\bG_x$, $x\in X$ are finite and
  \begin{equation*}
    \left(\forall x\in X\right)
    \left(\exists\text{ nbhd }V_x\ni x\right)
    \left(\forall g\in \bG\setminus \bG_x\right)
    \quad:\quad
    V_xg\cap V_x=\emptyset.
  \end{equation*}
  By \cite[Th\'eor\`eme 5.3.1]{zbMATH03192990} and (a supported variant of) \cite[Proposition 5.2.3]{zbMATH03192990} we have a spectral sequence
  \begin{equation*}
    E_2^{p,q}
    :=
    H^p\left(\bG,\ H^q_{\Supp(\Phi\bG)}(\cF)\right)
    \xRightarrow{\quad}
    H^{p+q}_{\Phi/\bG}(\cF/\bG),
  \end{equation*}
  and the two softness assumptions yield the conclusion. 
\end{proof}

Naturally, the cohomology vanishing constraint \Cref{eq:isotr.gp.coh.0} is automatic in the context of \cite[Lemma 2.4]{zbMATH00825443}, where all sheaves in sight are $\bR$-linear: higher finite-group cohomology valued in $\bQ$-vector spaces vanishes \cite[Proposition 9.40]{rot}.

\addcontentsline{toc}{section}{References}

\begin{thebibliography}{10}

\bibitem{zbMATH05904318}
H.~Abels, A.~Manoussos, and G.~Noskov.
\newblock Proper actions and proper invariant metrics.
\newblock {\em J. Lond. Math. Soc., II. Ser.}, 83(3):619--636, 2011.

\bibitem{zbMATH03603988}
Herbert Abels.
\newblock A universal proper {G}-space.
\newblock {\em Math. Z.}, 159:143--158, 1978.

\bibitem{MR2883692}
Alejandro Adem and Michele Klaus.
\newblock Lectures on orbifolds and group cohomology.
\newblock In {\em Transformation groups and moduli spaces of curves}, volume~16
  of {\em Adv. Lect. Math. (ALM)}, pages 1--19. Int. Press, Somerville, MA,
  2011.

\bibitem{am_coh-gp}
Alejandro Adem and R.~James Milgram.
\newblock {\em Cohomology of finite groups.}, volume 309 of {\em Grundlehren
  Math. Wiss.}
\newblock Berlin: Springer, 2nd ed. edition, 2004.

\bibitem{zbMATH01992380}
S.~Antonyan and S.~de~Neymet.
\newblock Invariant pseudometrics on {Palais} proper {{\(G\)}}-spaces.
\newblock {\em Acta Math. Hung.}, 98(1-2):59--69, 2003.

\bibitem{zbMATH05808276}
Sergey~A. Antonyan.
\newblock Proper actions on topological groups: applications to quotient
  spaces.
\newblock {\em Proc. Am. Math. Soc.}, 138(10):3707--3716, 2010.

\bibitem{zbMATH05182649}
A.~V. Arhangelskii and V.~V. Uspenskij.
\newblock Topological groups: local versus global.
\newblock {\em Appl. Gen. Topol.}, 7(1):67--72, 2006.

\bibitem{stacks-project}
The Stacks~Project Authors.
\newblock Stacks project.

\bibitem{zbMATH00936741}
Zoltan~T. Balogh.
\newblock A small {Dowker} space in {ZFC}.
\newblock {\em Proc. Am. Math. Soc.}, 124(8):2555--2560, 1996.

\bibitem{bdv}
Bachir Bekka, Pierre de~la Harpe, and Alain Valette.
\newblock {\em Kazhdan's property ({T})}, volume~11 of {\em New Mathematical
  Monographs}.
\newblock Cambridge University Press, Cambridge, 2008.

\bibitem{bl_eq-shv}
Joseph Bernstein and Valery Lunts.
\newblock {\em Equivariant sheaves and functors}, volume 1578 of {\em Lect.
  Notes Math.}
\newblock Berlin: Springer-Verlag, 1994.

\bibitem{bs_stab_2002}
N.~P. Bhatia and G.~P. Szeg{\H o}.
\newblock {\em Stability theory of dynamical systems}.
\newblock Classics in Mathematics. Springer-Verlag, Berlin, 2002.
\newblock Reprint of the 1970 original [MR0289890 (44 \#7077)].

\bibitem{zbMATH05546218}
Joseph~E. Borzellino and Victor Brunsden.
\newblock A manifold structure for the group of orbifold diffeomorphisms of a
  smooth orbifold.
\newblock {\em J. Lie Theory}, 18(4):979--1007, 2008.

\bibitem{zbMATH06223376}
Joseph~E. Borzellino and Victor Brunsden.
\newblock The stratified structure of spaces of smooth orbifold mappings.
\newblock {\em Commun. Contemp. Math.}, 15(5):37, 2013.
\newblock Id/No 1350018.

\bibitem{bt_forms}
Raoul Bott and Loring~W. Tu.
\newblock {\em Differential forms in algebraic topology}, volume~82 of {\em
  Graduate Texts in Mathematics}.
\newblock Springer-Verlag, New York-Berlin, 1982.

\bibitem{bourb_top_en_1}
Nicolas Bourbaki.
\newblock Elements of mathematics. {General} topology. {Part} 1. {Translation}
  of the {French} original.
\newblock Actualites scientifiques et industrielles {Hermann}. {Adiwes}
  {International} {Series} in {Mathematics}. {Paris}: {Hermann}, {Editeurs} des
  {Sciences} et des {Arts}; {Reading}, {Mass}. etc.: {Addison}-{Wesley}
  {Publishing} {Company}. {VII}, 436 p. (1966)., 1966.

\bibitem{bred_cpct-transf}
Glen~E. Bredon.
\newblock {\em Introduction to compact transformation groups}, volume~46 of
  {\em Pure Appl. Math., Academic Press}.
\newblock Academic Press, New York, NY, 1972.

\bibitem{bred_shf_2e_1997}
Glen~E. Bredon.
\newblock {\em Sheaf theory}, volume 170 of {\em Grad. Texts Math.}
\newblock New York, NY: Springer, 2nd ed. edition, 1997.

\bibitem{zbMATH00825443}
Ulrich Bunke and Martin Olbrich.
\newblock Gamma-cohomology and the {Selberg} zeta function.
\newblock {\em J. Reine Angew. Math.}, 467:199--219, 1995.

\bibitem{2408.15053v3}
Alexandru Chirvasitu, Rafael Dahmen, Karl-Hermann Neeb, and Alexander
  Schmeding.
\newblock On the singularities of the exponential function of a semidirect
  product, 2024.
\newblock \url{http://arxiv.org/abs/2408.15053v3}.

\bibitem{eng_dim}
Ryszard Engelking.
\newblock {\em Dimension theory}.
\newblock North-Holland Publishing Co., Amsterdam-Oxford-New York; PWN---Polish
  Scientific Publishers, Warsaw, 1978.
\newblock Translated from the Polish and revised by the author, North-Holland
  Mathematical Library, 19.

\bibitem{eng_top_1989}
Ryszard Engelking.
\newblock {\em General topology.}, volume~6 of {\em Sigma Ser. Pure Math.}
\newblock Berlin: Heldermann Verlag, rev. and compl. ed. edition, 1989.

\bibitem{god_faisc_1958}
Roger Godement.
\newblock {\em Topologie alg\'ebrique et th\'eorie des faisceaux}, volume No.
  13 of {\em Publications de l'Institut de Math\'ematiques de l'Universit\'e{}
  de Strasbourg [Publications of the Mathematical Institute of the University
  of Strasbourg]}.
\newblock Hermann, Paris, 1958.
\newblock Actualit\'es Scientifiques et Industrielles, No. 1252. [Current
  Scientific and Industrial Topics].

\bibitem{zbMATH03192990}
A.~Grothendieck.
\newblock Sur quelques points d'alg{\`e}bre homologique.
\newblock {\em T{\^o}hoku Math. J. (2)}, 9:119--221, 1957.

\bibitem{MR271925}
Otomar H\'ajek.
\newblock Parallelizability revisited.
\newblock {\em Proc. Amer. Math. Soc.}, 27:77--84, 1971.

\bibitem{zbMATH03289133}
Richard~E. Hodel.
\newblock Sum theorems for topological spaces.
\newblock {\em Pac. J. Math.}, 30:59--65, 1969.

\bibitem{hm_pro-lie-bk}
Karl~H. Hofmann and Sidney~A. Morris.
\newblock {\em The {Lie} theory of connected pro-{Lie} groups. {A} structure
  theory for pro-{Lie} algebras, pro-{Lie} groups, and connected locally
  compact groups}, volume~2 of {\em EMS Tracts Math.}
\newblock Z{\"u}rich: European Mathematical Society (EMS), 2007.

\bibitem{hjjm_bdle}
D.~Husem\"{o}ller, M.~Joachim, B.~Jur\v{c}o, and M.~Schottenloher.
\newblock {\em Basic bundle theory and {$K$}-cohomology invariants}, volume 726
  of {\em Lecture Notes in Physics}.
\newblock Springer, Berlin, 2008.
\newblock With contributions by Siegfried Echterhoff, Stefan Fredenhagen and
  Bernhard Kr\"{o}tz.

\bibitem{hus_fib}
Dale Husemoller.
\newblock {\em Fibre bundles}, volume~20 of {\em Graduate Texts in
  Mathematics}.
\newblock Springer-Verlag, New York, third edition, 1994.

\bibitem{zbMATH05718348}
Patrick Iglesias, Yael Karshon, and Moshe Zadka.
\newblock Orbifolds as diffeologies.
\newblock {\em Trans. Am. Math. Soc.}, 362(6):2811--2831, 2010.

\bibitem{2301.05325v3}
Michael Kapovich.
\newblock A note on properly discontinuous actions, 2023.
\newblock \url{http://arxiv.org/abs/2301.05325v3}.

\bibitem{ks_shv-mfld}
Masaki Kashiwara and Pierre Schapira.
\newblock {\em Sheaves on manifolds. {With} a short history ``{Les} d{\'e}buts
  de la th{\'e}orie des faisceaux'' by {Christian} {Houzel}}, volume 292 of
  {\em Grundlehren Math. Wiss.}
\newblock Berlin etc.: Springer-Verlag, 1990.

\bibitem{kms_natop}
Ivan Kol{\'a}{\v{r}}, Peter~W. Michor, and Jan Slov{\'a}k.
\newblock {\em Natural operations in differential geometry}.
\newblock Berlin: Springer-Verlag, 1993.

\bibitem{ksz_trnsf}
J.~L. Koszul.
\newblock {\em Lectures on groups of transformations. {Notes} by {R}. {R}.
  {Simha} and {R}. {Sridharan}}, volume~32 of {\em Lect. Math. Phys., Math.,
  Tata Inst. Fundam. Res.}
\newblock Springer, Berlin; Tata Inst. of Fundamental Research, Bombay, 1965.

\bibitem{lee2013introduction}
John~M. Lee.
\newblock {\em Introduction to smooth manifolds}, volume 218 of {\em Graduate
  Texts in Mathematics}.
\newblock Springer, New York, second edition, 2013.

\bibitem{massey_basic-alg-top_1991}
William~S. Massey.
\newblock {\em A basic course in algebraic topology}, volume 127 of {\em Grad.
  Texts Math.}
\newblock New York etc.: Springer-Verlag, 1991.

\bibitem{MR1466622}
I.~Moerdijk and D.~A. Pronk.
\newblock Orbifolds, sheaves and groupoids.
\newblock {\em $K$-Theory}, 12(1):3--21, 1997.

\bibitem{mz}
Deane Montgomery and Leo Zippin.
\newblock {\em Topological transformation groups}.
\newblock Robert E. Krieger Publishing Co., Huntington, N.Y., 1974.
\newblock Reprint of the 1955 original.

\bibitem{most_sect}
Paul~S. Mostert.
\newblock Sections in principal fibre spaces.
\newblock {\em Duke Math. J.}, 23:57--71, 1956.

\bibitem{mnk}
James~R. Munkres.
\newblock {\em Topology}.
\newblock Prentice Hall, Inc., Upper Saddle River, NJ, 2000.
\newblock Second edition of [ MR0464128].

\bibitem{MR61821}
V.~V. Nemycki\u~i.
\newblock Topological problems of the theory of dynamical systems.
\newblock {\em Amer. Math. Soc. Translation}, 1954(103):85, 1954.

\bibitem{palais_slice-nc}
Richard~S. Palais.
\newblock On the existence of slices for actions of non-compact {Lie} groups.
\newblock {\em Ann. Math. (2)}, 73:295--323, 1961.

\bibitem{ragh}
M.~S. Raghunathan.
\newblock {\em Discrete subgroups of {L}ie groups}.
\newblock Ergebnisse der Mathematik und ihrer Grenzgebiete, Band 68.
  Springer-Verlag, New York-Heidelberg, 1972.

\bibitem{rot}
Joseph~J. Rotman.
\newblock {\em An introduction to homological algebra}.
\newblock Universitext. Springer, New York, second edition, 2009.

\bibitem{MR776636}
Mary~Ellen Rudin.
\newblock Dowker spaces.
\newblock In {\em Handbook of set-theoretic topology}, pages 761--780.
  North-Holland, Amsterdam, 1984.

\bibitem{thurst_3mfld}
William~P. Thurston.
\newblock {\em The geometry and topology of three-manifolds. {V}ol. {IV}}.
\newblock American Mathematical Society, Providence, RI, [2022] \copyright
  2022.
\newblock Edited and with a preface by Steven P. Kerckhoff and a chapter by J.
  W. Milnor.

\bibitem{td_alg-top}
Tammo tom Dieck.
\newblock {\em Algebraic topology}.
\newblock EMS Textb. Math. Z{\"u}rich: European Mathematical Society (EMS),
  2008.

\bibitem{weib_halg}
Charles~A. Weibel.
\newblock {\em An introduction to homological algebra}, volume~38 of {\em Camb.
  Stud. Adv. Math.}
\newblock Cambridge: Cambridge University Press, 1994.

\bibitem{wil_top}
Stephen Willard.
\newblock {\em General topology}.
\newblock Dover Publications, Inc., Mineola, NY, 2004.
\newblock Reprint of the 1970 original [Addison-Wesley, Reading, MA;
  MR0264581].

\bibitem{MR1299067}
G.~Willis.
\newblock The structure of totally disconnected, locally compact groups.
\newblock {\em Math. Ann.}, 300(2):341--363, 1994.

\bibitem{zbMATH00797049}
George Willis.
\newblock Totally disconnected groups and proofs of conjectures of {Hofmann}
  and {Mukherjea}.
\newblock {\em Bull. Aust. Math. Soc.}, 51(3):489--494, 1995.

\end{thebibliography}

\def\polhk#1{\setbox0=\hbox{#1}{\ooalign{\hidewidth
  \lower1.5ex\hbox{`}\hidewidth\crcr\unhbox0}}}

\Addresses

\end{document}